\documentclass[12pt]{iopart}
\usepackage{iopams}

\usepackage{color}
\usepackage[]{graphics}
\usepackage{verbatim}
\usepackage{eepic,epic}
\usepackage{epsfig,subfigure,epstopdf}
\usepackage{graphicx}
\usepackage{ntheorem}
  \usepackage[colorlinks,linktocpage,linkcolor=blue]{hyperref}

\newtheorem{theorem}{Theorem}[section]
\newtheorem{lemma}{Lemma}[section]

\newtheorem{remark}{Remark}[section]

\newtheorem{corollary}{Corollary}[section]
\newtheorem*{proof}{Proof}

\newcommand{\al}{\alpha}
\newcommand{\fy}{\varphi}

\renewcommand{\d}{{\rm d}}
\def\Dal{{\partial_t^\al}}

\def\Om{\Omega}
\def\II{(\Om)}

\def\dH#1{\dot H^{#1}(\Omega)}
\def\tildu{\tilde{u}^\delta}

\def\bDal{\bar\partial_\tau^\alpha}

\def\Fht{F_{h,\tau}}
\eqnobysec

\begin{document}

\title[Numerical Analysis of Backward Subdiffusion]{Numerical Analysis of Backward Subdiffusion Problems}

\author{Zhengqi Zhang and Zhi Zhou}

\address{Department of Applied Mathematics, The Hong Kong Polytechnic University, Kowloon, Hong Kong SAR, China}
\ead{19076082r@connect.polyu.hk and zhizhou@polyu.edu.hk}

\vspace{10pt}
\begin{indented}
\item[]April 2019
\end{indented}

\begin{abstract}
The aim of this paper is to develop and analyze numerical schemes for approximately solving the backward problem of subdiffusion equation
 involving a fractional derivative in time with order $\alpha\in(0,1)$.
After using quasi-boundary value method to regularize the "mildly" ill-posed problem, we propose a fully discrete scheme
by applying finite element method (FEM) in space and convolution quadrature (CQ) in time.
We provide a thorough error analysis of the resulting discrete system in both cases of smooth and nonsmooth data.
The analysis relies heavily on smoothing properties of (discrete) solution operators,
and nonstandard error estimate for the direct problem in terms of problem data regularity.
The theoretical results are useful to balance discretization parameters, regularization parameter and noise level.
Numerical examples are presented to illustrate  the theoretical results.
\end{abstract}

%

\section{Introduction}
Let $\Omega\subset\mathbb{R}^d $ ($d\ge1$) be a bounded and convex domain with smooth boundary
$\partial\Omega$, and consider the following subdiffusion equation
\begin{eqnarray}\label{eqn:fde-0}
\eqalign{\partial_t^\alpha u -\Delta u&=f\qquad {\rm in}\,\,\, \Omega\times(0,T) ,\\
\qquad \quad ~u&=0\qquad {\rm on}\,\,\, \partial \Omega\times(0,T),\\
\qquad ~~u(0)&=u_0\qquad {\rm in}\,\,\, \Omega}
\end{eqnarray}
where $T>0$ is a fixed terminal time, $f \in L^\infty(0,T;L^2(\Omega))$ and $u_0 \in L^2(\Omega)$ are given
source term and initial data, respectively, and $\Delta$ is the Laplace operator in space. Here
$\Dal u(t)$ denotes the Caputo fractional derivative in time $t$ of order $\alpha\in(0,1)$: 
\begin{eqnarray*}
   \Dal u(t)= \frac{1}{\Gamma(1-\alpha)}\int_0^t(t-s)^{-\alpha} u'(s)\d s.
\end{eqnarray*}

In recent years, there has been a growing interest in the mathematical and numerical analysis of subdiffusion models due to
their diverse applications in describing subdiffusion processes arising from physics, engineering, biology
and finance. In a subdiffusion process, the mean squared particle displacement grows only sublinearly with time,
instead of growing linearly with time as in the normal diffusion process. We refer interested readers to \cite{Metzler:2014,MetzlerKlafter:2000}
for a long list of applications of subdiffusion arising from biology and physics.

Inverse problems for fractional diffusion have attracted much interest, and there has already been a vast literature; see e.g.,
review papers \cite{JinRundell:2015, LiLiuYamamoto:2019b, LiYamamoto:2019a, LiuLiYamamoto:2019c} and references therein. In this paper, we aim at the classical backward problem, i.e., determining the function $u(x,t)$ with $(x,t)\in \Omega\times[0,T)$ from a terminal
observation
\begin{equation*}
u(x,T)=g(x),\qquad \mbox{for all} \quad x\in\Omega.
\end{equation*}

With $f\equiv0$, the subdiffusion
model \eref{eqn:fde-0} has the following smoothing property \cite{SakamotoYamamoto:2011}:
\begin{equation*}
c_1\| u_0 \|_{L^2\II} \le \| u(T) \|_{H^2\II} \le c_2 \| u_0 \|_{L^2\II}.
\end{equation*}
This property contrasts sharply with the classical parabolic counterpart ($\alpha=1$), whose solution is infinitely differentiable
in space for all $t>0$. Thus, the backward problem of subdiffusion is far ``less'' ill-posed than that of normal diffusion.
The existence, uniqueness and stability of the time-fractional backward problem were analyzed
by Sakamoto and Yamamoto in \cite{SakamotoYamamoto:2011}. This work motivates many subsequent developments of regularized algorithms.
In \cite{LiuYamamoto:2010}, Liu and Yamamoto proposed a numerical method based on the quasi-reversibility method,
and analyze the approximation error (in terms of noise level) under \textit{a priori} smoothness assumption on $u_0$.
Then a total variation regularization method was proposed and studied by Wang and Liu in \cite{WangLiu:TV}.
In \cite{WangWei:iter}, Wang and Wei developed and analyzed an iteration method to regularize the backward problem.
The quasi-boundary value method for solving the fractional backward problem
was firstly studied in \cite{YangLiu:2013} for a one-dimensional subdiffusion model,
and then extended  in \cite{WeiWang:2014} to the general case by modifying the regularization term.
See also \cite{Liu:2019} for a novel H\"older type estimate of the quasi-boundary value methods.

To solve the regularized system, people applied different numerical approaches, e.g., finite element method,
finite different method, etc. Then some discretization error will be introduced into the system.
Therefore it is necessary to establish an estimate to
balance discretization parameter, regularization parameter and noise level.
However, such an analysis remains unavailable, and it is precisely this gap that the project aims to fill in.

Specifically, we assume that the observation data $g^\delta$ is noisy such that
\begin{equation*}
 \| g^\delta - g \|_{L^2\II} \le \delta.
\end{equation*}
To regularize the ill-posed problem, we apply the quasi-boundary value method \cite{Liu:2019, YangLiu:2013} and
consider
\begin{eqnarray}\label{eqn:quasi-boundary}
 \eqalign{
~\quad \partial_t^\alpha \tilde u^\delta -\Delta \tilde u^\delta   &= f.\qquad\mbox{in}\,\,\,\Omega\times(0,T) ,\\
  \qquad \quad\qquad \tilde u^\delta &=0\qquad\mbox{on}\,\,\,\partial \Omega\times(0,T),\\
\gamma \tilde u^\delta(0) + \tilde u^\delta(T)&= g_\delta\qquad\mbox{in}\,\,\,\Omega,\\
}
\end{eqnarray}
where $\gamma>0$ denotes the regularization parameter. In \cite{YangLiu:2013}, Yang and Liu considered the homogeneous problem ($f\equiv0$). {It was proved
that the regularized problem
\eref{eqn:quasi-boundary} has a unique solution, and if $u_0 \in L^2\II$, then for all $t\in[0,T]$  there holds
\begin{eqnarray}\label{eqn:back-err-cont-nonsm-0}
\|(\tilde u^\delta - u)(t)  \|_{L^2\II} \rightarrow 0, \quad\mbox{as}\quad \gamma,\delta\rightarrow 0~~\mbox{and} ~~\frac\delta\gamma\rightarrow0.
\end{eqnarray}
Moreover, if $u_0\in \mbox{Dom}(A) = H^2\II\cap H_0^1\II$, there holds 
\begin{eqnarray*}
\|  (\tilde u^\delta - u)(t)  \|_{C([0,T];L^2\II)}\le c \delta \gamma^{-1},
\end{eqnarray*}
where the constant $c$ depends only on $u_0$, $g$, $g_\delta$, but is independent
of $\delta$ and $\gamma$. By choosing $\gamma=O(\sqrt{\delta})$ \textit{a priori}, one obtains an approximation
with accuracy $O(\sqrt{\delta})$.} The result contrasts sharply with {that for}
normal diffusion, and the proof relies on the linear-decay property
of the Mittag-Leffler function $E_{\alpha,1}(-x)$.

To numerically solve the backward subdiffusion problem, we discretize the regularized problem (\ref{eqn:quasi-boundary})
by applying piecewise linear finite element method (FEM) in space and convolution quadrature generated by backward Euler scheme (CQ-BE) in time.
We provide thorough error analysis of proposed scheme and specify the way to balance the discrization error, regularization parameter and noise level.
For example, we let $h$ be the spatial mesh size and $\tau$ be temporal step size. Suppose $u(t)$ is the exact solution of the backward subdiffusion problem
and $\tilde U^\delta_n$ is the fully discrete solution (which approximates the exact solution at $t_n=n\tau$).
Then we prove that if $u_0 \in L^2\II$, then with $\ell_h = \max(1,|\ln h|)$, there holds (Theorem \ref{thm:fully-error-rev} (ii))
\begin{eqnarray*}
 \hskip-20pt \|\tilde U_0^\delta-u(0)\|_{L^2\II} \rightarrow 0,\quad \mbox{as}~~\gamma\rightarrow0, ~\frac\delta\gamma\rightarrow0,
  ~ \frac{h\ell_h^{\frac12}}{\gamma}\rightarrow0~\mbox{and}~ \frac{\tau^{\frac12}}{\gamma}\rightarrow0.
\end{eqnarray*}
Besides, for $n\ge 1$, there holds (Theorem \ref{thm:fully-error-rev})
\begin{eqnarray*}
\hskip-50pt\| \tilde U_n^\delta - u(t_n)  \|_{L^2\II} \le c\Big( \gamma t_{n}^{-\alpha}+\big(\delta+
  \gamma^{-1}(h^2 \ell_h +\tau)\big)\min(\gamma^{-1},t_n^{-\alpha}) + \gamma^{-1}\tau t_n^{\alpha-1}\Big).
\end{eqnarray*}
where the constant $c$ is independent of $h$, $\tau$, $\gamma$ and $\delta$.
Then by choosing \textit{a priori} $h\ell_h^{\frac12}=O(\delta)$, $\tau=O(\delta^2)$ and $\gamma=O(\delta)$, one obtains an approximation with accuracy $O(\delta)$
for all $t\in(0,T]$, even though the approximation at $t=0$ has no convergence rate.
The analysis relies heavily on smoothing properties of (discrete) solution operators,
and nonstandard error estimate for the direct problem in terms of problem data regularity
\cite{JinLazarovZhou:SIAM2013, JinLiZhou:2017sisc, JinLiZhou:var}.
Such the estimates could be improved provided that the problem data is smoother and compatible with the boundary condition.
For instance, if $u_0\in H^2\II \cap H_0^1\II$, there holds (Theorem \ref{thm:fully-error-rev} (i))
\begin{eqnarray*}
\hskip-0.5in\| \tilde U_n^\delta - u(t_n)  \|_{L^2\II} \le c\left\{\eqalign{
& \gamma+(h^2+\tau+\delta)\min(\gamma^{-1}, t_n^{-\alpha})+
\tau t_{n}^{\alpha-1}  , n\ge1; \\
&\gamma+(h^2+\tau+\delta)\gamma^{-1},\qquad n=0.}\right.
\end{eqnarray*}
{As far as we know, this is the first work providing rigorous error analysis of numerical methods for solving
the time-fractional backward problem.}

The rest of the paper is organized as follows. In Section \ref{sec:prelim}, we provide some preliminary results
about the solution representation and the regularization at the continuous level,
which will be intensively used in error estimation. Then in Section \ref{sec:semi} and Section \ref{sec:fully}, we describe and analyze spatially semi-discrete
scheme and fully discrete scheme,  {respectively}. Finally, in Section \ref{sec:numerics},
we present illustrative  numerical examples to illustrate the theoretical analysis.
Throughout, the notation $c$ denotes a generic constant, which may change
at each occurrence, but it is always independent of the noise level $\delta$, the regularization parameter $\gamma$,
the mesh size $h$ and time step size $\tau$ etc.

\section{Preliminary}\label{sec:prelim}

\subsection{Solution representation and Mittag-Leffler functions}
In this section, we recall the representation of the solution to the subdiffusion problem \eref{eqn:fde-0},
which plays a key role in the analysis.

To begin with, we introduce some notation.
For $q\ge0$, we denote by $\dH q$ the Hilbert space induced by the norm:
\begin{equation*}
  \|v\|_{\dH q}^2=\sum_{j=1}^{\infty}\lambda_j^q ( v,\fy_j )^2
\end{equation*}
with $\{\lambda_j\}_{j=1}^\infty$ and $\{\fy_j\}_{j=1}^\infty$ being respectively the eigenvalues and
the $L^2(\Omega)$-orthonormal eigenfunctions of the negative Laplacian $-\Delta$ on the domain
$\Omega$ with a homogeneous Dirichlet boundary condition.
Then $\{\fy_j\}_{j=1}^\infty$ forms orthonormal basis in $L^2(\Omega)$.
Further, $\|v\|_{\dH 0}=\|v\|_{L^2(\Omega)}=(v,v)^{1/2}$ is the norm in $L^2(\Omega)$.
Besides, it is easy to verify that $\|v\|_{\dH 1}= \|\nabla v\|_{L^2(\Omega)}$ is also the norm in $H_0^1(\Om)$
and  $\|v\|_{\dH 2}=\|\Delta v\|_{L^2(\Omega)}$ is equivalent to the norm in $H^2(\Om)\cap H^1_0(\Om)$
\cite[Section 3.1]{Thomee:2006}.

Now we represent the solution to problem \eref{eqn:fde} using the eigenpairs $\{\lambda_j,\fy_j\}_{j=1}^\infty$.
To this end, we define solution operators $F(t)$ and $E(t)$ from \cite{JinLazarovZhou:SIAM2013},
\begin{eqnarray}\label{eqn:op}
\hskip-0.75in F(t)\chi=\sum_{j=1}^\infty E_{\al,1}(-\lambda_jt^\al)(\chi,\fy_j)\fy_j\quad \mbox{and}\quad
E(t)\chi=\sum_{j=1}^\infty t^{\al-1}E_{\al,\al}(-\lambda_jt^\al)(\chi,\fy_j)\fy_j,\,\,\,\,
\end{eqnarray}
where $E_{\alpha,\beta}(z)$ is the two-parameter Mittag-Leffler function: 
\begin{equation}\label{eqn:ML}
  E_{\alpha,\beta}(z) = \sum_{k=0}^\infty \frac{z^{k}}{\Gamma(k\alpha+\beta)}\quad \forall z\in \mathbb{C}.
\end{equation}
Then the solution of the forward problem \eref{eqn:fde-0} could be written as
\begin{equation}\label{eqn:sol-rep}
u(t)=F(t)u_0+\int_0^t E(t-s)f(s)ds.
\end{equation}

The Mittag-Leffler function $E_{\alpha,\beta}(z)$ is a generalization of the familiar exponential function
$e^z$ appearing in normal diffusion. 
Then following decay behavior of $E_{\alpha,\beta}(z)$ is crucial
to the smoothing properties of $F(t)$ and $E(t)$: for any $\alpha\in (0,1)$, the function $E_{\alpha,1}
(-\lambda t^\alpha)$ decays only polynomially like $t^{-\alpha}$ as $t\to\infty$ (cf. Lemma \ref{lem1}),
which contrasts sharply with the exponential decay for $e^{-\lambda t}$
appearing in normal diffusion.

Note that $w(t) = E_{\alpha,1}(-\lambda t^\alpha)$ be the solution to the initial value problem
\begin{equation*}
\Dal w (t) + \lambda w(t) = 0,\quad \mbox{with}\quad w(0)=1.
\end{equation*}
By means of Laplace transform, it can be written as
\begin{equation}\label{eqn:sol-ode}
w(t)=E_{\alpha,1}(-\lambda t^\alpha) = \frac{1}{2\pi\mathrm{i}} \int_{\Gamma_{\theta,\sigma}} e^{zt} z^{\alpha-1}(z^\alpha+\lambda)^{-1}\,dz.
\end{equation}
with integration over a contour $\Gamma_{\theta,\sigma}$ in the complex plane $\mathbb{C}$
(oriented counterclockwise), defined by
\begin{equation*}
  \Gamma_{\theta,\sigma}=\left\{z\in \mathbb{C}: |z|=\delta, |\arg z|\le \theta\right\}\cup
  \{z\in \mathbb{C}: z=\rho e^{\pm\mathrm{i}\theta}, \rho\ge \sigma\} .
\end{equation*}
Throughout, we fix $\theta \in(\frac{\pi}{2},\pi)$ so that $z^{\al} \in \Sigma_{\al,\theta}
\subset \Sigma_{\theta}:=\{0\neq z\in\mathbb{C}: {\rm arg}(z)\leq\theta\},$ for all $z\in\Sigma_{\theta}$.

The next lemma provides the upper and lower bounds of Mittag--Leffler functions \eref{eqn:ML}. See {\cite[Theorem 4]{Simon:2014}} for detailed proof.
\begin{lemma}\label{lem1}
Assume that $\alpha\in(0,1)$. Then there holds that
\begin{equation*}
 \frac{1}{1+\Gamma(1-\alpha)z}\leq E_{\al,1}(-z)\leq \frac{1}{1+\Gamma(1+\alpha)^{-1}z},\quad \mbox {for all}\,\,~ z\ge 0.
\end{equation*}
\end{lemma}

These important bounds directly translate into the limited smoothing
property in both space and time for the solution operators $E(t)$ and $F(t)$.
Next, we state a few regularity results. The proof of these results can be found in, e.g.,
 \cite{Bajlekov:2001, SakamotoYamamoto:2011, JinLiZhou:var}.

\begin{theorem} \label{thm:reg-u}
Let $u(t)$ be defined in \eref{eqn:sol-rep}. Then the following statements hold.
\begin{itemize}
  \item[$\rm(i)$] If $u_0 \in \dH q$ with $s\in[0,2]$ and $f=0$, then $u(t)$ is the solution to problem \eref{eqn:fde-0}, and
  $u(t)$ satisfies
\begin{equation*}
\|  \partial_t^{(m)}  u(t) \|_{\dH p} \le c t^{\frac{(s-p)\alpha}{2}-m} \| {u_0} \|_{\dH q}
\end{equation*}
with $0\le p-q\le 2$ and any integer $m\ge 0$ .
  \item[$\rm(ii)$]If $u_0=0$ and $f\in L^p(0,T;L^2(\Omega))$ with $1<p<\infty$, then there holds
\begin{equation*}
\|  u\|_{L^p(0,T;\dot H^2(\Omega))}
+\|\Dal u\|_{L^p(0,T;L^2(\Omega))}
\le c\|f\|_{L^p(0,T;L^2(\Omega))}.
\end{equation*}
Moreover, if $f\in L^p(0,T;L^2(\Omega))$ with $1/\alpha<p<\infty$, then $u(t)$ is the solution to problem \eref{eqn:fde-0}
such that $u\in C([0,T];L^2\II)$.
\end{itemize}
\end{theorem}

\subsection{Reformulation of original problem}
In our paper, we shall study an equivalent reformulation of the original backward subdiffusion problem \eref{eqn:fde-0}.
We let  $w(t)=u(t)-\int_0^t E(t-s)f(s) \,ds$, then $w$ satisfies the subdiffusion problem \eref{eqn:fde-0} with trivial source term,
and the terminal data is
$$ w(T) = u(T) - \int_0^T E(T-s)f(s) \,ds.  $$
Meanwhile, in case that $f\in L^p(0,T;L^2(\Omega))$ with $1/\alpha<p<\infty$, then by Theorem \ref{thm:reg-u} we have
 $w(0)=u(0)=u_0$. Then without loss of generality,
we only consider the following backward subdiffusion problem with trivial source data:
\begin{eqnarray}\label{eqn:fde}
\eqalign{
\partial_t^\alpha u -\Delta u &= 0
\qquad\mbox{in}\,\,\,\Omega\times(0,T) ,\\
\qquad \quad~   u &=0\qquad \mbox{on}\,\,\,\partial \Omega\times(0,T),\\
\qquad ~ u(T)&= g  \qquad\mbox{in}\,\,\,\Omega.\\
}
\end{eqnarray}
The solution $u$ has the representation that
\begin{equation}\label{eqn:sol-back}
 u(t) = F(t) u(0) = F(t) (F(T)^{-1} g).
\end{equation}

Inspired by the quasi-boundary value method discussed in \cite{Liu:2019, YangLiu:2013},
we defined an axillary function $\tilde u(t)$, which satisfies the regularized problem (without noise):
\begin{eqnarray}\label{eqn:tildeu}
\eqalign{
\Dal\tilde{u}(t)-\Delta \tilde{u}(t)&=0,\qquad \mbox{in} \,\,\,\Om\times(0,T),\\
\qquad \qquad \quad  \tilde u &=0 \qquad\mbox{on}\,\,\,\partial \Omega\times(0,T),\\
~~\gamma\tilde{u}(0)+\tilde{u}(T)&=g,\qquad \mbox{in}\,\,\, \Om.
}
\end{eqnarray}
Here $\gamma$ denotes the regularization parameter. The appearance of regularization term essentially improves the regularity of the backward problem.

Analogue to \eref{eqn:sol-back}, the function $\tilde u$ can be represented by
\begin{equation}\label{eqn:sol-back-tilde}
\hskip-35pt\tilde u(t) = F(t)\tilde u(0) = F(t) \big(\gamma I + F(T)\big)^{-1} g = F(t) \big(\gamma I+ F(T)\big)^{-1} \big(F(T) u_0\big),
\end{equation}
where $I$ denotes the identity operator.

The next lemma provides an estimate of the operator $F(t) \big(\gamma I + F(T)\big)^{-1}$.
\begin{lemma}\label{lem:tildeu-0}
Let $F(t)$ be operator defined in \eref{eqn:op}, then
\begin{equation*}
\|F(t) \big(\gamma {I} + F(T)\big)^{-1}v\|_{\dH{q}}\le c \min(\gamma^{-1},t^{-\alpha}) \| v\|_{\dH{q}}\quad \forall ~ q\ge0,
\end{equation*}
where the generic constant $c$ may depends on $T$, but is always independent of $\gamma$ and $t$.
\end{lemma}
\begin{proof}
From Lemma \ref{lem1} we have $E_{\al,1}(-z)>0$ for any $z\ge 0$, then
\begin{equation*}
\eqalign{
\|F(t) \big(\gamma{I} + F(T)\big)^{-1}v\|_{\dH{q}}^2
&=\sum_{j=1}^{{\infty}}\left[\frac{E_{\al,1}(-\lambda_jt^\al)}{\gamma+E_{\al,1}(-\lambda_j T^\al)}\right]^2\lambda_j^q(v,\fy_j)^2.}
\end{equation*}
By applying the fact that $0 \le E_{\al,1}(-z) \le 1$ with $z\ge0$, we arrive at
\begin{equation*}
\|F(t) \big(\gamma{I} + F(T)\big)^{-1}v\|_{\dH{q}}^2 \le \gamma^{-1}\|v\|_{\dH{q}}^2.
\end{equation*}
On the other hand, we apply Lemma \ref{lem1} again to obtain for any $t\in(0,T]$
{\begin{eqnarray*}
\frac{E_{\al,1}(-\lambda_jt^\al)}{\gamma+E_{\al,1}(-\lambda_j T^\al)}
& \le \frac{E_{\al,1}(-\lambda_jt^\al)}{ E_{\al,1}(-\lambda_j T^\al)}   \le \frac{1+\Gamma(1-\al)(\lambda_jT^\al)}{1+\Gamma(1+\al)^{-1}(\lambda_jt^\al)}\\
& \le 1+ \frac{ \Gamma(1-\al)(\lambda_jT^\alpha)}{\Gamma(1+\al)^{-1}(\lambda_jt^\al)}\le c_T t^{-\al}.
\end{eqnarray*}}
and hence
\begin{equation*}
\|F(t) \big(\gamma{I} + F(T)\big)^{-1}v\|_{\dH{q}} \le c t^{-\alpha}\|v\|_{\dH{q}}.
\end{equation*}
This completes the proof of the lemma.
\end{proof}
%
%

Using this lemma, we can derive the following estimate of $\tilde u(t) - u(t)$ with $t\in[0,T)$.
\begin{lemma}\label{lem:err-tildeu}
Let $u$ and $\tilde u$ be solutions to problems \eref{eqn:fde} and \eref{eqn:tildeu}, respectively. Then
there holds
\begin{equation*}
\|\tilde{u}(0)-u(0)\|_{L^2\II}\le c\gamma^{\frac q2} \|u_0\|_{\dH{q}} \qquad  \forall ~q\in[0,2] .
\end{equation*}
Meanwhile, for any $t\in(0,T)$, there holds
\begin{equation*}
\|\tilde{u}(t)-u(t)\|_{L^2\II}\leq c\gamma t^{-(1-\frac{q}{2})\al}\|u_0\|_{\dH q}\qquad  \forall ~q\in[0,2] .
\end{equation*}
where the generic constant $c$ may depends on $T$, but is always independent of $\gamma$ and $t$.
\end{lemma}
\begin{proof}
By \eref{eqn:sol-back} and \eref{eqn:sol-back-tilde} we obtain
\begin{equation*}
  \tilde{u}(0)-u(0) =- \big(\gamma I+F(T)\big)^{-1} \gamma u_0.
\end{equation*}
Now applying \eref{eqn:op} and positivity of $E_{\al,1}(z)$ with $z\le0$,
we derive that for any $q\in[0,2]$,
\begin{equation*}
\eqalign{
\|\tilde{u}(0)-u(0)\|_{L^2\II}^2&=\|(\gamma I+F(T))^{-1}\gamma u_0\|_{L^2\II}^2\\
&= \sum_{j=1}^\infty \Big(\frac{\gamma}{\gamma+E_{\al,1}(-\lambda_jT^\al)}\Big)^2(u_0,\fy_j)^2\\
&\leq \sum_{j=1}^\infty \frac{\gamma^q}{\lambda_j^q |E_{\al,1}(-\lambda_jT^\al)|^q} \lambda_j^q(u_0,\fy_j)^2.}
\end{equation*}
The property of Mittag-Leffler functions in Lemma \ref{lem1} implies that
\begin{equation*}
 \frac{\gamma^q}{\lambda_j^q|E_{\al,1}(-\lambda_jT^\al)|^q} \le \frac{c\gamma^q(1+\lambda_j T)^q}{\lambda_j^q} \le c_T \gamma^q.
\end{equation*}
and hence
\begin{equation*}
\|\tilde{u}(0)-u(0)\|_{L^2\II}^2\leq c \gamma^q \|u_0\|_{\dH q}^2.
\end{equation*}
Now we turn to the second estimate, which follows from the representation
\begin{equation*}
\tilde{u}(t)-u(t) = - F(t) \big(\gamma I+F(T)\big)^{-1} \gamma u_0.
\end{equation*}
Here we apply the definition of the solution operator and obtain
\begin{equation*}
\eqalign{
\|\tilde{u}(t)-u(t)\|_{L^2\II}^2&=\|{F(t)}(\gamma I+F(T))^{-1}\gamma u_0\|_{L^2\II}^2\\
 &=   \sum_{j=1}^\infty \left(\frac{\gamma E_{\al,1}(-\lambda_j t^\al)}{\gamma+E_{\al,1}(-\lambda_j T^\al)}\right)^2 (u_0,\fy_j)^2\\
&\le\gamma^2 \sum_{j=1}^\infty \left(\frac{  E_{\al,1}(-\lambda_j t^\al)}{ \lambda_j^{q/2} E_{\al,1}(-\lambda_j T^\al)}\right)^2 \lambda_j^q (u_0,\fy_j)^2}
\end{equation*}
Then Lemma \ref{lem1} leads to the estimate
\begin{equation*}
 \frac{  E_{\al,1}(-\lambda_j t^\al)}{ \lambda_j^{{q/2}} E_{\al,1}(-\lambda_j T^\al)}\le \frac{c(1+\lambda_jT^\alpha)}{\lambda_j^{q/2}(1+\lambda_jt^\alpha)}
 \le c_T \frac{\lambda_j^{1-q/2}}{1+\lambda_jt^\alpha}\le c_T t^{-(1-q/2)\alpha}, \\
\end{equation*}
and therefore there holds
\begin{equation*}
\|\tilde{u}(t)-u(t)\|_{L^2\II}^2 \le c \gamma^2 t^{-(2-q)\alpha} \| u_0 \|_{\dH q}^2.\\
\end{equation*}
This completes the proof of the lemma.
\end{proof}

If $u_0\in L^2\II = \dH0$, the preceding lemma does not imply a convergence rate. However, one can still show the convergence
in case of nonsmooth data.
\begin{corollary}\label{cor:err-tildeu-L2}
Assume that $u_0\in L^2\II$. Let $u$ and $\tilde u$ be solutions to problems \eref{eqn:fde} and \eref{eqn:tildeu}, respectively. Then
there holds that
\begin{equation*}
\lim_{\gamma\rightarrow0}\|\tilde{u}(0) -  u(0)\|_{L^2\II} =0 .
\end{equation*}
\end{corollary}
\begin{proof}
In case that $u_0\in L^2\II$, we know that $\tilde u, u\in C([0,T];L^2\II)$. Then for any small $\epsilon$, we choose $t_0$ small enough such that
$$  \|  \tilde u(t_0) - \tilde u(0) \|_{L^2\II}  + \|  u(t_0) -  u(0) \|_{L^2\II} < \epsilon/2. $$
Then by Lemma \ref{lem:err-tildeu}, we may find $\gamma_0$ small enough such that
$$  \|  \tilde u(t_0)  -  u(t_0) \|_{L^2\II} < \epsilon/2\quad \mbox{for all} ~\gamma<\gamma_0. $$
By triangle inequality , we obtain that for any $\gamma<\gamma_0$
$$  \|  \tilde u(0)  -  u(0) \|_{L^2\II} < \epsilon. $$
Therefore, $\tilde u(0)$ converges to $u(0)$ in $L^2$-sense, as $\gamma \rightarrow 0$.
\end{proof}

{\begin{remark}\label{rem:err-tildeu-L2}
The estimate in Corollary \ref{cor:err-tildeu-L2} seems to be a special case of (\ref{eqn:back-err-cont-nonsm-0}) in case that $\delta=0$.
However, the  proof of \cite[Theorem 3.4]{YangLiu:2013} is not directly applicable in this case. Besides,
the estimate of $\| \widetilde u - u \|$ is missing in the literature, but it is important
in the error analysis of the numerical solution in the next section.
\end{remark}}

\section{Spatial semidiscrete method by finite element method}\label{sec:semi}
In this section, we shall propose and analyze a spatially semidiscrete scheme for solving the backward subdiffusion problem (\ref{eqn:fde}).
{Even though the semidiscrete scheme is not directly implementable and rarely used in practical computation, it is important for
understanding the role of the regularity of problem data and also for the analysis of fully discrete schemes.}

\subsection{Semidiscrete scheme for solving direct problem.}
Now we describe the spatial discretization by finite element method.
For $h\in(0,h_0]$,  we denote
by $\mathcal{T}_h = \{K_j\}$ a triangulation of $\Omega_h=$ {Int}$(\cup \overline K_j)$
into mutually disjoint open face-to-face simplices $K_j$. Assume that all boundary vertices of $\Omega_h$
locate on $\partial\Omega$. We also assume that
$\{\mathcal{T}_h\}$  is globally quasi-uniform, i.e., $|K_j|\ge c h^d$ with a given $c>0$.
Let $X_h$ be the finite dimensional space of continuous piecewise
linear functions associated with $\mathcal{T}_h$, that vanish outside $\Omega_h$.

The semidiscrete Galerkin FEM for problem \eref{eqn:fde-0}
is: find $ u_h (t)\in X_h$ such that
\begin{equation}\label{fem}
 \eqalign{
{( \Dal u_{h},\chi)}+ (\nabla u_h,\nabla \chi) &= {(f, \chi)},
\qquad \forall ~\chi\in X_h,\ T \ge t >0,\\
 \hskip1.2in u_h(0)&=P_h u_0,}
\end{equation}
To describe the schemes, we need the $L^2(\Omega)$ projection $P_h:L^2(\Omega)\to X_h$ and
Ritz projection $R_h:\dH1\to X_h$, respectively, defined by (recall that $(\cdot, \cdot)$
denotes the $L^2(\Omega)$ inner product)
\begin{eqnarray*}
 \hskip.3in     (P_h \psi,\chi) & =(\psi,\chi) \qquad\forall~ \chi\in X_h,\psi\in L^2(\Omega),\\
    (\nabla R_h \psi,\nabla\chi) & =(\nabla \psi,\nabla\chi) \qquad \forall ~\chi\in X_h, \psi\in \dot H^1(\Omega).
\end{eqnarray*}
The following approximation properties of $R_h$ and $P_h$ are well known \cite[Chapter 1]{Thomee:2006}:
\begin{eqnarray}
\fl\|P_h\psi-\psi\|_{L^2(\Omega)}+h\|\nabla(P_h\psi-\psi)\|_{L^2(\Omega)}& \leq ch^q\|\psi\|_{H^q(\Omega)}\qquad \forall\psi\in \dH q, q=1,2,\label{eqn:err-Ph}\\
\fl\|R_h\psi-\psi\|_{L^2(\Omega)}+h\|\nabla(R_h\psi-\psi)\|_{L^2(\Omega)}& \leq ch^q\|\psi\|_{H^q(\Omega)}\qquad \forall\psi\in \dH q, q=1,2.\label{eqn:err-Rh}
\end{eqnarray}

Upon introducing the discrete Laplacian $\Delta_h: X_h\to X_h$ defined by
\begin{equation*}
  -(\Delta_h\psi,\chi)=(\nabla\psi,\nabla\chi)\quad\forall\psi,\,\chi\in X_h,
\end{equation*}
and $f_h= P_h f$, we may write the spatially semidiscrete problem \eref{fem} as
\begin{equation}\label{fem-operator}
   \Dal u_{h}(t)-\Delta_h u_h(t) =f_h(t) \mbox{ for }t\ge0 \quad \mbox{with} \quad  u_h(0)=P_h u_h.
\end{equation}
Now we give a representation of the solution of \eref{fem-operator} using the
eigenvalues and eigenfunctions
$\{\lambda^h_j\}_{j=1}^K$ and $\{\varphi_j^h\}_{j=1}^K$ of the discrete Laplacian
$-\Delta_h$. Here
we introduce the discrete analogue of \eref{eqn:op} for $t>0 $:
\begin{equation}\label{eqn:sol-op-semi-0}
\fl F_h(t)\chi=\sum_{j=1}^K E_{\al,1}(-\lambda_j^h t^\al)(\chi,\fy_j^h)\fy_j^h~~\mbox{and}~~
E_h(t)\chi=\sum_{j=1}^K t^{\al-1} E_{\al,\al}(-\lambda_j^h t^\al)(\chi,\fy_j^h)\fy_j^h,
\end{equation}
Then the solution $u_h(t)$ of the semidiscrete problem \eref{fem-operator} can be expressed by:
\begin{equation}\label{Duhamel_o}
     u_h(t)= F_h(t) {u_h(0)} + \int_0^t E_h(t-s) f_h(s)\,\d s.
\end{equation}

The discrete solution operator $E_h(t)$ satisfies the following smoothing property. See {\cite[Lemma 3.2]{JinLazarovZhou:SIAM2013}}
for proof.
\begin{lemma}\label{lem4}
We have $E_h(t)$ and $\psi\in S_h$. Then we have for all $t>0$ and $q\in[0,1]$
\begin{equation*}
\|\Delta_h^q E_h(t) \psi\|_{L^2\II}\leq c t^{(1-q)\alpha -1} \|  \psi \|_{L^2\II}.
\end{equation*}
\end{lemma}

\subsection{Semidiscrete scheme for solving backward problem.}
In this part, we consider the semidiscrete solution $\tilde u_h^\delta(t) \in S_h$ such that
\begin{equation}\label{DBWP}
\eqalign{
\Dal\tildu_h(t)-\Delta_h\tildu_h(t)&=0,\quad \forall t\in(0,T]\\
\hskip0.22in\gamma\tildu_h(0)+\tildu_h(T)&=P_hg_\delta.}
\end{equation}
Then the function {$\tilde{u}_h^\delta$} can be written as
\begin{equation}\label{eqn:semi-sol-1}
 \tilde{u}_h^\delta(t) = F_h(t) {  \tilde{u}^\delta_h(0)} = F_h(t) (\gamma + F_h(T))^{-1} P_hg_\delta.
\end{equation}
Meanwhile, we shall use an axillary function  $\tilde{u}_h(t)$, which is the semidiscrete solution to \eref{eqn:tildeu}, i.e.,
satisfying
\begin{equation}\label{DBWP-1}
\eqalign{
\Dal\tilde{u}_h(t)-\Delta_h\tilde{u}_h(t)&=0,\quad \forall t\in(0,T]\\
\hskip0.22in\gamma\tilde{u}_h(0)+\tilde{u}_h(T)&=P_hg,}
\end{equation}
Similarly, we have the representation
\begin{equation}\label{eqn:semi-sol-2}
\tilde u_h(t) = F_h(t) \tilde u_h(0) = F_h(t) (\gamma I+ F_h(T))^{-1} P_hg.
\end{equation}
%

Analogue to Lemma \ref{lem:tildeu-0}, we have the following estimate of the operator $F_h(t) (\gamma I+ F_h(T))^{-1}$.
Note that the error is independent of the mesh size $h$.
\begin{lemma}\label{lem:tildeuh-0}
Let $F_h(t)$ be operator defined in \eref{eqn:sol-op-semi-0}, then there holds that
\begin{equation*}
\|F_h(t) \big(\gamma I+ F_h(T)\big)^{-1}v\|_{L^2\II}\le c \min(\gamma^{-1},t^{-\alpha}) \| v\|_{L^2\II}\quad \forall ~ v\in X_h,
\end{equation*}
where the constant $c$ may depends on $T$, but is always independent of $h$, $\gamma$ and $t$.
\end{lemma}

This Lemma together with \eref{eqn:semi-sol-1} and \eref{eqn:semi-sol-2}
immediately leads to the following estimate of $\tilde{u}_h^\delta(t)-\tilde{u}_h(t)$.
\begin{corollary}\label{cor:err-2}
Let $\tilde{u}_h^\delta$ and $\tilde{u}_h$ be the solution to the semidiscrete problems \eref{DBWP} and \eref{DBWP-1},
respectively. Then, there holds that
\begin{equation*}
\| (\tilde{u}_h^\delta -\tilde{u}_h)(t) \|_{L^2\II} \leq c \delta \min(\gamma^{-1}, t^{-\alpha}) \quad \forall~t\in [0,T],
\end{equation*}
where the generic constant $c$ is independent of $\gamma$, $\delta$, $h$ and $t$.
\end{corollary}

Next, we shall derive a bound of $\tilde{u}_h -\tilde{u}$.
\begin{lemma}\label{lem:err-smooth}
Assume that $u_0\in \dH2$. Let $\tilde{u}$ be the solution to the regularized backward subdiffusion problem \eref{eqn:tildeu},
and $\tilde{u}_h$ be the solution to the corresponding semidiscrete problem \eref{DBWP-1}. Then there holds
\begin{equation*}
\|(\tilde{u}_h -\tilde{u})(t) \|_{L^2\II} \leq ch^2\min(\gamma^{-1},t^{-\alpha}) \| u_0 \|_{\dH 2} \quad \forall~t\in [0,T],
\end{equation*}
where $c$ might depend on $T$, but is always independent of $h$, $\gamma$ and $t$.
\end{lemma}
\begin{proof}
We split $\tilde{u}_h(t)-\tilde{u}(t)$ into two components such that
\begin{equation*}
\tilde{u}_h(t)-\tilde{u}(t)=(\tilde{u}_h(t)-R_h\tilde{u}(t))+(R_h\tilde{u}(t)-\tilde{u}(t))=: { \zeta(t)} + \rho(t),
\end{equation*}
By the approximation property of the Ritz projection in \eref{eqn:err-Rh}, we have
\begin{equation}\label{eqn:rho-bound-1}
 \| \rho(t)\|_{L^2\II} \le c h^2 \|  \tilde{u}(t) \|_{\dH2} \le c h^2  \| u_0 \|_{\dH2} .
\end{equation}
where the last inequality follows from  \eref{eqn:sol-back-tilde} and Lemma \ref{lem:tildeu-0} (with $t=T$).

Now we turn to the bound of $\zeta  = \tilde{u}_h - R_h\tilde{u}$, where $\tilde{u}_h$ and $R_h\tilde{u}$ satisfy
\begin{equation*}
\gamma \tilde{u}_h(0)+\tilde{u}_h(T)=P_h g \qquad\mbox{and}\qquad
\gamma R_h\tilde{u}(0)+R_h\tilde{u}(T)=R_hg,
\end{equation*}
respectively. By noting the fact $\Delta_h R_h=P_h\Delta$, we have
\begin{equation}\label{eqn:err-0}
\Dal\zeta(t)-\Delta_h\zeta(t) = -P_h \Dal\rho(t) \quad\mbox{with}\quad \gamma\zeta(0)+\zeta(T) = (P_h-R_h) g.
\end{equation}
Then we arrive at
\begin{equation*}
\zeta(T)=F_h(T)\zeta(0)-\int_0^T E_h(T-s)P_h\partial_s^\al\rho(s)ds.
\end{equation*}
We add $\gamma\zeta(0)$ at both sides of the equation and use  \eref{eqn:err-0} to derive that
\begin{equation*}
 P_hg-R_hg=(\gamma I+F_h(T))\zeta(0)-\int_0^T E_h(T-s)P_h\partial_s^\al\rho(s)ds,
\end{equation*}
and therefore
\begin{eqnarray*}
 \zeta(t) &= F_h(t)\big(\gamma I+F_h(T)\big)^{-1}\Big[{(P_h-R_h)}g +\int_0^T E_h(T-s)P_h\partial_s^\al\rho(s)ds\Big]\\
 &\qquad
{ - \int_0^t E_h(t-s)P_h\partial_s^\al\rho(s)ds} \\&=:I_1+I_2+I_3.
\end{eqnarray*}
The properties  \eref{eqn:err-Ph} and  \eref{eqn:err-Rh}, and Lemma \ref{lem:tildeuh-0} lead to the estimate that
\begin{eqnarray*}
 \| I _1 \|_{L^2\II} &\le c \min(\gamma^{-1},t^{-\alpha}) \|{(P_h-R_h)}g\|_{L^2\II} \le c h^2 \min(\gamma^{-1},t^{-\alpha}) \| g\|_{\dH2}\\ &\le c h^2 \min(\gamma^{-1},t^{-\alpha}) \| u_0 \|_{\dH2}.
\end{eqnarray*}
The last inequality is the direct result of the solution regularity in Theorem \ref{thm:reg-u}.
Similarly, we apply Lemmas \ref{lem4} and \ref{lem:tildeuh-0}, and stability of $L^2$ projection $P_h$ to arrive at
\begin{equation*}
\| I_2 \|_{L^2\II} \le c \min(\gamma^{-1},t^{-\alpha})   \int_0^T (T-s)^{\alpha-1} \| \partial_s^\al\rho(s)\|_{L^2\II} ds  .
\end{equation*}
Then \eref{eqn:err-Rh} and the solution regularity in Theorem \ref{thm:reg-u} immediately imply that
\begin{eqnarray*}
\| I_2 \|_{L^2\II}&\le c h^2 \min(\gamma^{-1},t^{-\alpha}) \int_0^T (T-s)^{\alpha-1} \| \partial_s^\al u(s) \|_{\dH2}\,ds\\
 &\le c h^2 \min(\gamma^{-1},t^{-\alpha}) \int_0^T (T-s)^{\alpha-1} s^{-\alpha}\,ds\|  u_0 \|_{\dH2}\\
 &\le c h^2 \min(\gamma^{-1},t^{-\alpha})\|  u_0 \|_{\dH2}.
\end{eqnarray*}
Similar argument also leads to a bound of the term $I_3$:
\begin{eqnarray*}
\| I_3\|_{L^2\II} &\le c h^2   \int_0^T (T-s)^{\alpha-1} \| \partial_s^\al u(s) \|_{\dH2}\,ds\\
 &\le c h^2 \int_0^T (T-s)^{\alpha-1} s^{-\alpha}\,ds\|  u_0 \|_{\dH2}\le c h^2  \|  u_0 \|_{\dH2}.
\end{eqnarray*}
As a result, we arrive at the desired estimate.
\end{proof}

Then, Lemmas \ref{lem:err-tildeu} and \ref{lem:err-smooth} and Corollary \ref{cor:err-2} together lead to the following theorem which providing an error estimate of the numerical solution $\widetilde u_h^\delta$,
in case of smooth initial data, i.e., $u_0\in D(\Delta)=\dH2$.

\begin{theorem}\label{thm:err-smooth}
Assume that $u_0\in \dH2$.
Let $u$ be the solution to the problem \eref{eqn:fde} and $\widetilde u_h^\delta$ be the solution to the (regularized)  semidiscrete problem \eref{DBWP}.
Then there holds
\begin{equation*}
\|\tildu_h(t)-u(t)\|_{L^2\II}\le c \big(\gamma +(h^2 + \delta)\min(\gamma^{-1},t^{-\alpha})\big)\quad \forall~t\in [0,T],
\end{equation*}
where $c$ might depend   on $T$ and $u_0$, but is always independent of  $h$, $\gamma$, $\delta$ and $t$.
\end{theorem}

{
\begin{remark}\label{rem:rate-semi-smooth}
The error estimate in Theorem \ref{thm:err-smooth} is useful, since
it specifies the scale to balance the discrization
error, regularization parameter and noise level. For example, if we decide
the \textit{a priori} choice of parameters: $h = O(\sqrt\delta)$ and $\gamma= O(\sqrt\delta)$, then there holds
$$ \| \tilde{u}_h^\delta(0) - u_0   \|_{L^2\II} \le c\sqrt{\delta}.  $$
On the other hand, for any $t>0$, we have
$$ \| \tilde{u}_h^\delta(t) - u(t)   \|_{L^2\II} \le c\delta t^{-\alpha},  $$
by the \textit{a priori} choice of parameters: $h=O(\sqrt\delta)$ and $ \gamma ={O(\delta)}$. This is the first study of the discretized problem,
and the result is consistent with the estimate in the continuous level, see e.g. \cite[Theorem 3.4]{YangLiu:2013}.
The analysis relies heavily on the nonstandard error estimate for the direct problem in terms of problem data regularity
\cite{JinLazarovZhou:SIAM2013}.
\end{remark}}

Next, we shall consider the worse case that $u_0\in L^2\II$.
\begin{lemma}\label{lem:err-nonsmooth}
Assume that $u_0\in L^2\II$. Let $\tilde{u}$ be the solution to the regularized backward subdiffusion problem \eref{eqn:tildeu},
and $\tilde{u}_h$ be the solution to the corresponding semidiscrete problem \eref{DBWP-1}. Then there holds for all $t\in[0,T]$ and $\ell_h=\max(1,|\ln h|)$
\begin{equation*}
\|(\tilde{u}_h -\tilde{u})(t) \|_{L^2\II} \leq  c \gamma^{-1}\min(\gamma^{-1},t^{-\alpha}) h^{2} \ell_h \| u_0 \|_{L^2\II},
\end{equation*}
where the constant $c$ might depend on $T$, but is always independent of $h$, $\gamma$ and $t$.
\end{lemma}
\begin{proof}
By using the $L^2$-projection $P_h$, we split
 $\tilde{u}_h(t)-\tilde{u}(t)$ into two components:
\begin{equation*}
\tilde{u}_h(t)-\tilde{u}(t)=(\tilde{u}_h(t)-P_h\tilde{u}(t))+(P_h\tilde{u}(t)-\tilde{u}(t))=: \zeta(t) + \rho(t),
\end{equation*}
By the approximation property of the $L^2$-projection in \eref{eqn:err-Rh}, we have
\begin{equation*}
 \|\rho(t)\|_{L^2\II} \le c h^2 \|  \tilde{u}(t) \|_{\dH{2}}\le c_T h^2 \gamma^{-1} \| u_0 \|_{L^2\II},
\end{equation*}
where the last inequality follows from
the solution representation \eref{eqn:sol-back-tilde}, Lemma \ref{lem:tildeu-0}  and Theorem \ref{thm:reg-u}, such that
\begin{equation}\label{eqn:est-01}
\|  \tilde{u}(t) \|_{\dH{2}}\le c \gamma^{-1} \| F(T) u_0 \|_{\dH{2}} \le c \gamma^{-1} T^{-\alpha} \| u_0 \|_{L^2\II}.
\end{equation}

Now we turn to the bound of $\zeta  = \tilde{u}_h - P_h\tilde{u}$, where $\tilde{u}_h$ and $P_h\tilde{u}$ satisfy
\begin{equation*}
\gamma \tilde{u}_h(0)+\tilde{u}_h(T)=P_h g \qquad\mbox{and}\qquad
\gamma P_h\tilde{u}(0)+ P_h\tilde{u}(T)=P_hg,
\end{equation*}
respectively. By noting the fact $\Delta_h R_h=P_h\Delta$, we have
\begin{equation}\label{eqn:err-1}
\Dal\zeta(t)-\Delta_h\zeta(t) = \Delta_h(P_h-R_h)  \widetilde u(t) \quad\mbox{with}\quad \gamma\zeta(0)+\zeta(T) = 0.
\end{equation}
Then we arrive at
\begin{equation*}
\zeta(T)=F_h(T)\zeta(0)+\int_0^T E_h(T-s)\Delta_h(P_h-R_h)  \widetilde u(s)ds.
\end{equation*}
We add $\gamma\zeta(0)$ at both sides of the equation and derive that
\begin{equation*}
\zeta(0) = -(\gamma I+F_h(T))^{-1}\int_0^T E_h(T-s) \Delta_h (P_h -R_h) \widetilde u(s)ds,
\end{equation*}
and hence
\begin{eqnarray*}
\zeta(t) &=F_h(t)\zeta(0)  +\int_0^t E_h(t-s)\Delta_h(P_h-R_h)  \widetilde u(s)ds\\
& = -F_h(t) \big(\gamma I+F_h(T)\big)^{-1} \int_0^T E_h(T-s) \Delta_h (P_h -R_h) \widetilde u(s)ds\\
&\qquad { + \int_0^t E_h(t-s)\Delta_h(P_h-R_h)  \widetilde u(s)ds}\\
& =: I_1 + I_2.
\end{eqnarray*}
Similarly, we apply Lemmas \ref{lem4} and \ref{lem:tildeuh-0}, to arrive at
\begin{eqnarray*}
\| I_1 \|_{L^2\II} &\le c   \min(\gamma^{-1},t^{-\alpha})   \int_0^T (T-s)^{\alpha\epsilon-1} \| \Delta_h^\epsilon(P_h -R_h) \widetilde u(s) \|_{L^2\II}  ds\\
&\le  c  \min(\gamma^{-1},t^{-\alpha})  h^{-2\epsilon}  \int_0^T (T-s)^{\alpha\epsilon-1} \| (P_h -R_h)\widetilde u(s) \|_{L^2\II}  ds
\end{eqnarray*}
where we apply the inverse estimate for FEM functions in the second inequality.  The approximation properties \eref{eqn:err-Rh} and \eref{eqn:err-Ph} lead to
\begin{equation*}
\| I_1 \|_{L^2\II}  \le  c   \min(\gamma^{-1},t^{-\alpha})  h^{2-2\epsilon}  \int_0^T (T-s)^{\alpha\epsilon-1} \| \widetilde u(s) \|_{\dH2}  ds,
\end{equation*}
and then the regularity estimate of $\widetilde u$ in \eref{eqn:est-01} implies that
\begin{eqnarray*}
\| I_1 \|_{L^2\II} &\le  c   \gamma^{-1}\min(\gamma^{-1},t^{-\alpha})  h^{2-2\epsilon}  \int_0^T (T-s)^{\alpha\epsilon-1} T^{-\alpha}  ds  \| u_0 \|_{L^2\II}\\
&\le c\gamma^{-1}\min(\gamma^{-1},t^{-\alpha})  h^{2-2\epsilon} \epsilon^{-1} \| u_0 \|_{L^2\II}.
\end{eqnarray*}
Similar argument also leads to a bound of the term $I_2$:
\begin{eqnarray*}
\| I_2 \|_{L^2\II} &\le  c h^{2-2\epsilon} \epsilon^{-1} \| u_0 \|_{L^2\II}.
\end{eqnarray*}
Then the desired assertion follows immediately by choosing $\epsilon=1/\ell_h$.
\end{proof}

Then, Lemmas \ref{lem:err-tildeu} and \ref{lem:err-nonsmooth} and Corollary \ref{cor:err-2}
together lead to the following error estimate,
in case of nonsmooth initial data.
\begin{theorem}\label{thm:nonsmooth}
Assume that $u_0\in \dH{q}$ with $q\in[0,2]$.
Let $u$ be the solution to the problem \eref{eqn:fde} and $\widetilde u_h^\delta$ be the solution to the
(regularized) semidiscrete problem \eref{DBWP}.
Then there holds for all $t\in[0,T]$ and $\ell_h=\max(1,|\ln h|)$
\begin{equation*}
\hskip-0.7in\|\tildu_h(t)-u(t)\|_{L^2\II}\le c\Big( \min(\gamma^{q/2},\gamma t^{-(1-q/2)\alpha})
+(\gamma^{-(1-q/2)}h^2\ell_h^{1-q/2} + \delta)\min(\gamma^{-1},t^{-\alpha})\Big).
\end{equation*}
where the constant $c$ depends on $T$ and $u_0$, but is always independent of $h$, $\gamma$, $\delta$ and $t$.
\end{theorem}

{
\begin{remark}\label{rem:weak-semi}
In case that $u_0\in L^2\II$, the above estimate does not imply a convergence rate of $\tildu_h(0)$.
However, we can still show the convergence, provided suitable scales of parameters. The proof is a direct result of Corollaries \ref{cor:err-tildeu-L2} and \ref{cor:err-2},
and  Lemma  \ref{lem:err-nonsmooth}.

Let $u$ be the solution to the problem \eref{eqn:fde} and $\widetilde u_h^\delta$ be the solution to the semidiscrete problem \eref{DBWP}.
Then there holds
\begin{equation*}
 \|\tildu_h(0)-u(0)\|_{L^2\II} \rightarrow 0,\quad \mbox{as}~~\gamma\rightarrow0, ~\frac\delta\gamma\rightarrow0~\mbox{and}
  ~ \frac{h\ell_h^{1/2}}{\gamma}\rightarrow0.
\end{equation*}
\end{remark}}


\section{Fully discrete solution and error estimate}\label{sec:fully}

\subsection{Fully discrete scheme and solution operators.}
Now we study the time discretization of problem \eref{eqn:fde}.
We divide the time interval $[0,T]$ into a uniform grid, with $ t_n=n\tau$, $n=0,\ldots,N$, and $\tau=T/N$ being
the time step size. In case that $\fy(0)=0$, we approximate the Riemann-Liouville fractional derivative
\begin{equation*}
^{RL}\Dal \varphi(t)=\frac{1}{\Gamma(1-\alpha)}\frac{\d}{\d t}\int_0^t(t-s)^{-\alpha}\varphi(s)\d s
\end{equation*}
by the backward Euler (BE) convolution quadrature (with $\varphi_j=\varphi(t_j)$) \cite{Lubich:1986,JinLiZhou:2017sisc}:
\begin{equation*}
  ^{RL}\Dal \varphi(t_n) \approx \tau^{-\alpha} \sum_{j=0}^nb_j\varphi_{n-j}:=\bar\partial_\tau^\alpha \varphi_n,\quad\mbox{ with } \sum_{j=0}^\infty b_j\xi^j = (1-\xi)^\alpha.
\end{equation*}

The fully discrete scheme for problem \eref{eqn:fde-0} reads: find ${U_n}\in S_h$ such that
\begin{equation}\label{eqn:fully}
\bDal (U_n-U_0)-\Delta_h U_n= P_h f(t_n),\quad n=1,2,\ldots,N,
\end{equation}
with the initial condition $U_0=P_h u_0\in S_h$.
{Here we use the relation between Riemann-Liouville and Caputo fractional derivatives  \cite[p. 91]{KilbasSrivastavaTrujillo:2006}:
$$ \Dal u (t_n) = \Dal (u(t_n) - u_0) = {^{RL}\Dal} (u(t_n) - u_0) \approx \bar \partial_\tau^\alpha (u(t_n) - u_0). $$
}
 By means of discrete Laplace transform, the fully discrete solution $U_n\in S_h$ is given by
\begin{equation}\label{eqn:Sol-expr-uhtau}
  U_n = F_{h,\tau}^n U_0 + \tau \sum_{k=1}^{n} E_{h,\tau}^{n-k} P_hf(t_k),  \quad n=1,2,\ldots,N,
\end{equation}
where the fully discrete operators $F_{h,\tau}^n$ and $E_{h,\tau}^n$ are respectively defined by (see e.g., \cite{JinLiZhou:2017sisc})
\begin{eqnarray}\label{eqn:FE_ht-0}
F_{h,\tau}^n &= \frac{1}{2\pi\mathrm{i}}\int_{\Gamma_{\theta,\sigma}^\tau } e^{zt_n} \delta_\tau(e^{-z\tau})^{\alpha-1}({ \delta_\tau(e^{-z\tau})^\alpha}-\Delta_h)^{-1}\,\d z ,\\
E_{h,\tau}^n &= \frac{1}{2\pi\mathrm{i}}\int_{\Gamma_{\theta,\sigma}^\tau } e^{zt_n} ({ \delta_\tau(e^{-z\tau})^\alpha}-\Delta_h)^{-1}\,\d z ,\label{op:disc}
\end{eqnarray}
with $\delta_\tau(\xi)=(1-\xi)/\tau$ and the contour
$\Gamma_{\theta,\sigma}^\tau :=\{ z\in \Gamma_{\theta,\sigma}:|\Im(z)|\le {\pi}/{\tau} \}$
(oriented with an increasing imaginary part).

The next lemma gives elementary properties of the kernel $\delta_\tau(e^{-z\tau})$ \cite[Lemma B.1]{JinLiZhou:2017sisc}.
\begin{lemma}\label{lem:delta}
For any $\theta\in (\pi/2,\pi)$, there exists $\theta' \in (\pi/2,\pi)$  and
positive constants $c,c_1,c_2$ $($independent of $\tau$$)$ such that for all $z\in \Gamma_{\theta,\sigma}^\tau$
\begin{equation*}
  \eqalign{
& c_1|z|\leq
|\delta_\tau(e^{-z\tau})|\leq c_2|z|,\qquad
\delta_\tau(e^{-z\tau})\in \Sigma_{\theta'}. \\
& |\delta_\tau(e^{-z\tau})-z|\le c\tau |z|^{2},\qquad
 |\delta_\tau(e^{-z\tau})^\alpha-z^\alpha|\leq c\tau |z|^{1+\alpha}.
}
\end{equation*}
\end{lemma}

The fully discrete solution operators has been fully understood in  \cite{JinLiZhou:2017sisc}, by using the expression \eref{eqn:FE_ht-0} and \eref{op:disc}, resolvent estimate and properties of the kernel $\delta_\tau(e^{-z\tau})$
in Lemma \ref{lem:delta}. With the spectral decomposition, we can write
\begin{equation}\label{eqn:Sol-expr-uhtau-2}
  U_n = F_{h,\tau}^n U_0 = \sum_{j=1}^K F_{\tau}^n(\lambda_j^h) (u_0,\fy_j^h){\varphi_j^h}
\end{equation}
where $F_{\tau}^n(\lambda_j^h)$ is the solution to the discrete initial value problem
$$\bDal [F_{\tau}^n(\lambda_j^h)-F_{\tau}^0(\lambda_j^h)] + \lambda_j^h F_{\tau}^n(\lambda_j^h) = 0,\quad \mbox{with}\quad F_{\tau}^0(\lambda_j^h) = 1.$$
{From \eref{eqn:FE_ht-0}}, we know that ${F_{\tau}^n(\lambda_j^h)}$ could be written as
\begin{equation}\label{eqn:FE_ht}
F_{\tau}^n(\lambda_j^h) = \frac{1}{2\pi\mathrm{i}}\int_{\Gamma_{\theta,\sigma}^\tau } e^{zt_n} \delta_\tau(e^{-z\tau})^{\alpha-1}({ \delta_\tau(e^{-z\tau})^\alpha}+\lambda_j^h)^{-1}\,\d z.
\end{equation}

\begin{lemma}\label{lem:fully-sol}
Let  $F_{\tau}^n(\lambda)$ be defined as in \eref{eqn:FE_ht}. Then for $\lambda>0$, there holds
\begin{equation}\label{eqn:es-01}
\big|E_{\alpha,1}(-\lambda t_n^\alpha) - F_{\tau}^n(\lambda)\big|\le \frac{c}{(1+\lambda t_n^\alpha)}n^{-1}.
\end{equation}
  Meanwhile, there holds
\begin{equation}\label{eqn:es-02}
\lambda^{-1} \big|E_{\alpha,1}(-\lambda t_n^\alpha) - F_{\tau}^n(\lambda)\big|\le c\tau t_{n}^{\alpha-1}.
\end{equation}
 where $c$ is a generic number independent of $\lambda$, $t$ and $\tau$.
\end{lemma}

\begin{proof}
It has been proved in \cite{Jin:positive} that
$$\big|E_{\alpha,1}(-\lambda t_n^\alpha) - F_{\tau}^n(\lambda)\big|\le c n^{-1}.$$
Therefore it suffices to show that
$$\big|E_{\alpha,1}(-\lambda t_n^\alpha) - F_{\tau}^n(\lambda)\big|\le c \lambda^{-1} t_n^{-\alpha} n^{-1}.$$
From \eref{eqn:FE_ht} and \eref{eqn:sol-ode}, we know
\begin{equation*}
 \eqalign{
&\big|E_{\alpha,1}(-\lambda_j^h t^\alpha) - F_{\tau}^n(\lambda_j^h)\big|
\leq \big|\frac{1}{2\pi i}\int_{\Gamma_{\theta,\sigma}\backslash\Gamma_{\theta,\sigma}^\tau}e^{zt_n}z^{\al-1}(z^\alpha+\lambda)^{-1} dz\big| \\
&+\big|\frac{1}{2\pi i}\int_{\Gamma_{\theta,\sigma}^\tau}e^{zt_n}\left[z^{\al-1}(z^\al+\lambda)^{-1}-\delta_\tau(e^{-z\tau})^{\alpha-1}(\delta_\tau(e^{-z\tau})^\alpha+\lambda)^{-1}\right]dz  \big|\\
&=:I_1+I_2.
}
\end{equation*}
First of all, we shall establish a bound of $I_1$, which follows from the direct calculation:
\begin{equation*}
  \eqalign{
I_1&\leq c \int_{\Gamma_{\theta,\sigma}\backslash\Gamma_{\theta,\sigma}^\tau} |e^{zt_n}| |z|^{\al-1} |z^\alpha+\lambda|^{-1} \, |d z|
\le c \lambda^{-1} \int_{\pi/\tau\sin\theta}^\infty  e^{\rho(\cos\theta)t_n} \rho^{\alpha-1} d\rho \\
&\le c \lambda^{-1} t_n^{-\alpha}\int_{cn}^\infty  e^{-c \rho} \rho^{\alpha-1} d\rho \le  c \lambda^{-1} t_n^{-\alpha} n^{-1} \int_{cn}^\infty  e^{-c \rho} \rho^{\alpha} d\rho
\le c \lambda^{-1} t_n^{-\alpha} n^{-1}  .
}
\end{equation*}
Next we turn to $I_2$ . By lemma \ref{lem:delta}, we have for all $z\in \Gamma_{\theta,\sigma}^\tau$
\begin{equation*}
  \eqalign{
&\Big|\frac{z^{\al-1}}{z^\al+\lambda}-\frac{\delta_\tau(e^{-z\tau})^{\alpha-1}}{\delta_\tau(e^{-z\tau})^\alpha+\lambda}\Big|\\
=&\Big|\frac{z^{\al-1}\delta_\tau(e^{-z\tau})^{\alpha-1}(\delta_\tau(e^{-z\tau})-z)}{(z^\al+\lambda)(\delta_\tau(e^{-z\tau})^\alpha+\lambda)}\Big|
 +\Big|\frac{(z^{\al-1} - \delta_\tau(e^{-z\tau})^{\alpha-1})\lambda}{(z^\al+\lambda)(\delta_\tau(e^{-z\tau})^\alpha+\lambda)}\Big| \\\le& c\tau \lambda^{-1} |z|^{\alpha}.
}
\end{equation*}
Therefore, with $\sigma=t_n^{-1}$, the  term $I_2$ can be bounded as
\begin{equation*}
  \eqalign{
I_2&\leq c \tau \lambda^{-1} \int_{\Gamma_{\theta,\sigma}^\tau} |e^{zt_n}| |z|^{\al}  \,|dz| \\
&\le c\tau\lambda^{-1} \Big(\int_{\sigma}^{\infty} e^{-c\rho t_n} \rho^\alpha \,d\rho+ \sigma^{1+\alpha} \int_{-\theta}^\theta \, d \psi\Big)\\
&\le c\tau\lambda^{-1} t_n^{-\alpha-1} \le c\lambda^{-1} t_n^{-\alpha} n^{-1}.
}
\end{equation*}

Next, we turn to the estimate \eref{eqn:es-02}, which can be derived from the expressions:
\begin{equation*}
\eqalign{
&E_{\alpha,1}(-\lambda t_n^\alpha) = 1 - \frac{ \lambda}{2\pi\mathrm{i}}\int_{\Gamma_{\theta,\sigma}} e^{zt_n}
z^{-1} (z^\alpha+\lambda)^{-1}\,\d z, \\
&F_{\tau}^n(\lambda) = 1 -  \frac{ \lambda}{2\pi\mathrm{i}}\int_{\Gamma_{\theta,\sigma}^\tau } e^{zt_n}
\delta_\tau(e^{-z\tau})^{-1}({ \delta_\tau(e^{-z\tau})^\alpha}+\lambda)^{-1}\,\d z,
}
\end{equation*}
with $n\ge1$.
Then we arrive at
\begin{equation*}
  \eqalign{
&\lambda^{-1} \big|E_{\alpha,1}(-\lambda t_n^\alpha) - F_{\tau}^n(\lambda)\big|\\
\le& \big|\frac{1}{2\pi i}\int_{\Gamma_{\theta,\sigma}^\tau}e^{zt_n}\left[z^{-1}(z^\al+\lambda)^{-1}-
\delta_\tau(e^{-z\tau})^{-1}(\delta_\tau(e^{-z\tau})^\alpha+\lambda)^{-1}\right]dz  \big| \\
\le&  \big|\frac{1}{2\pi i}\int_{\Gamma_{\theta,\sigma}\backslash\Gamma_{\theta,\sigma}^\tau}e^{zt_n}z^{-1}(z^\alpha+\lambda)^{-1} dz\big|
=:II_1+II_2.
}
\end{equation*}
By Lemma \ref{lem:delta}, we have for all $z\in \Gamma_{\theta,\sigma}^\tau$
\begin{equation*}
  \eqalign{
\left|z^{-1}(z^\al+\lambda)^{-1}-
\delta_\tau(e^{-z\tau})^{-1}(\delta_\tau(e^{-z\tau})^\alpha+\lambda)^{-1}\right|
\le c\tau |z|^{-\alpha},
}
\end{equation*}
and therefore with the setting $\sigma=t_n^{-1}$ we have the bound  for $n\ge1$
\begin{equation*}
  \eqalign{
\fl II_1 &\leq c \tau   \int_{\Gamma_{\theta,\sigma}^\tau} |e^{zt_n}| |z|^{-\al}  \,|dz|
\le c\tau \Big(\int_{\sigma}^{\infty} e^{-c\rho t_n} \rho^{-\alpha} \,d\rho+ \sigma^{1-\alpha} \int_{-\theta}^\theta \, d \psi\Big)\le c\tau t_n^{\alpha-1}.
}
\end{equation*}
Similarly, to bound $II_2$, we apply Lemma \ref{lem:delta} to derive that  for $n\ge1$
\begin{equation*}
  \eqalign{
II_2 &\leq  c \int_{\Gamma_{\theta,\sigma}\backslash\Gamma_{\theta,\sigma}^\tau} |e^{zt_n}| |z|^{-\al-1} \, |d z|
\le c   \int_{\pi/\tau\sin\theta}^\infty  e^{\rho(\cos\theta)t_n} \rho^{-\alpha-1} d\rho \\
&\le c  t_n^{\alpha}\int_{cn}^\infty  e^{-c \rho} \rho^{-\alpha-1} d\rho \le  c t_n^{\alpha} n^{-1} \int_{0}^\infty  e^{-c \rho} \rho^{-\alpha} d\rho
\le c  t_n^{\alpha} n^{-1}\le c  \tau t_n^{\alpha-1} .
}
\end{equation*}
Both the estimates together with the fact that $E_{\alpha,1}(0) = F_{\tau}^0(\lambda)=1$ lead to the desired result.
 \end{proof}

The above lemma and Lemma \ref{lem1}  lead to the following corollary.
\begin{corollary}\label{cor:fully-sol}
For any $1\leq n\leq N$, $\Fht^n(\lambda)$ is positive, and
there exist positive constants $c_0$, $c_1$ such that
\begin{equation*}
\frac{c_0}{1+\lambda t_n^\alpha }\leq F_{\tau}^n(\lambda_j^h)\leq \frac{c_1}{1+\lambda  t_n^\alpha},
\end{equation*}
\end{corollary}

Then the next {corollary} follows immediately.
\begin{corollary}\label{cor:fully-sol-2}
Let $\Fht^n(\lambda)$ be defined as  \eref{eqn:FE_ht}, then there holds
\begin{equation*}
| \Fht^n(\lambda) \big(\gamma + \Fht^N(\lambda)\big)^{-1} |\le c \min(\gamma^{-1},t_n^{-\alpha}),
\end{equation*}
where the generic constant $c$ may depends on $T$, but is always independent of $\gamma$, $\lambda$, $\tau$, $n$ and $h$.
\end{corollary}

\begin{proof}
By Corollary \ref{cor:fully-sol}, we know that  $0 \le F_{\tau}^n (\lambda) \le c_1$, we arrive at
\begin{equation*}
| \Fht^n(\lambda) \big(\gamma + \Fht^N(\lambda)\big)^{-1} | \le c \gamma^{-1}.
\end{equation*}
On the other hand, we apply Corollary \ref{cor:fully-sol} again to obtain
\begin{equation*}
\frac{\Fht^n(\lambda)}{\gamma + \Fht^N(\lambda)} \le \frac{\Fht^n(\lambda)}{\Fht^N(\lambda)}
\le \frac{c(1+\lambda T^\alpha)}{1+\lambda t_n^\alpha} \le c_T t_n^{-\alpha}.
\end{equation*}
This completes the proof of the {corollary}.
\end{proof}

\subsection{fully discrete scheme for backward problem and error estimate.} Now we shall propose a fully discrete scheme for
solving the backward subdiffusion problem. Here we apply
the semidiscrete scheme and the convolution quadrature generated by backward Euler scheme. Then the  fully discrete scheme reads:
find $\tilde U_n^\delta\in X_h$, $n=1,2,\ldots,N$, such that

\begin{equation}\label{eqn:fully-back}
\eqalign{
\bDal (\tilde U_n^\delta - \tilde U_0^\delta)-\Delta_h\tilde U_n^\delta&=0,\quad \forall ~n=1,2,\ldots,N.\\
\hskip0.8in \gamma\tilde U_0^\delta+\tilde U_N^\delta&=P_hg_\delta.}
\end{equation}
Then the solution could be written as
\begin{equation}\label{eqn:tildeUdn}
  \tilde U_n^\delta = F_{h,\tau}^n \tilde{U}_0^\delta = F_{h,\tau}^n (\gamma I + F_{h,\tau}^N)^{-1} P_hg_\delta
= \sum_{j=1}^K \frac{F_{\tau}^n(\lambda_j^h)}{\gamma  + F_{\tau}^N(\lambda_j^h)} (P_h g_\delta, \fy_j^h){\varphi_j^h}.
\end{equation}

Similarly, we shall use the auxiliary solution $\tilde U_n$ satisfying
\begin{equation}\label{eqn:fully-back-2}
 \eqalign{
\bDal (\tilde U_n - \tilde U_0)-\Delta_h\tilde U_n&=0,\quad \forall ~n=1,2,\ldots,N.\\
\hskip0.8in \gamma\tilde U_0 +\tilde U_N &=P_hg.}
\end{equation}
Then $\tilde U_n$ could be written as
\begin{equation}\label{eqn:tildeUn}
 \tilde U_n = F_{h,\tau}^n (\gamma I + F_{h,\tau}^N)^{-1} P_hg
= \sum_{j=1}^K \frac{F_{\tau}^n(\lambda_j^h)}{\gamma  + F_{\tau}^N(\lambda_j^h)} ({P_hg}, \fy_j^h){\varphi_j^h}.
\end{equation}

The same as Corollary \ref{cor:err-2}, we may show the following estimate of $\tilde U_n^\delta - \tilde U_n$.
\begin{lemma}\label{lem:fully-lem1}
Let $\tilde U_n^\delta$ and $\tilde U_n $ be solutions to \eref{eqn:fully-back} and \eref{eqn:fully-back-2}, respectively.
Then there holds that
$$\| \tilde U_n^\delta - \tilde U_n  \|_{L^2\II} \le c \delta \min (\gamma^{-1}, t_n^{-\alpha}), \quad\mbox{for all}~ 0 \le n\le N.$$
where the generic constant $c$ is independent of  $\gamma$, $\delta$, $\tau$, $n$ and $h$.
\end{lemma}
\begin{proof}
From Corollary \ref{cor:fully-sol-2}, we have (${ \forall v\in X_h}$)
\begin{equation*}
\fl\|F_{h,\tau}^n \big(\gamma + F_{h,\tau}^N(T)\big)^{-1}v\|_{L^2\II}^2
=\sum_{j=1}^K \left[\frac{F_\tau^n(-\lambda_j^h)}{\gamma+F_\tau^N(-\lambda_j^h)}\right]^2 (v,\fy_j)^2\le c \min(\gamma^{-1},t_n^{-\alpha}) \|  v \|_{L^2\II}
\end{equation*}
Therefore for all $0 \le n\le N$
$$\| \tilde U_n^\delta - \tilde U_n  \|_{L^2\II} \le c  \min (\gamma^{-1}, t_n^{-\alpha}) \| g-g^\delta \| \le c \delta \min (\gamma^{-1}, t_n^{-\alpha}).$$
\end{proof}

\begin{lemma}\label{lem:fully-lem2}
Let  $\tilde U_n $  and $\tilde u_h(t)$ be solutions to  \eref{eqn:fully-back-2} and \eref{DBWP-1}, respectively.
Then there holds that
$$\| \tilde U_0 - \tilde u_h(0)  \|_{L^2\II} \le c \Big(\tau \gamma^{-1-(1-q/2)}\| u_0 \|_{\dH q} + h^2 \gamma^{-1} \| u_0 \|_{L^2\II}\Big).$$
\end{lemma}
where the generic constant $c$ is independent of  $\gamma$, $\delta$, $\tau$, $n$ and $h$.
\begin{proof}
By \eref{eqn:semi-sol-2}, we know the semidiscrete function $\tilde u_h(t)$ can be represented as
\begin{equation*}
\tilde u_h(0) = (\gamma I + F_h(T))^{-1} P_h g=\sum_{j=1}^K \frac{( g, \fy_j^h)}
{\gamma  + E_{\alpha,1}(-\lambda_j^hT^{\alpha})} {\varphi_j^h}.
\end{equation*}
This combined with \eref{eqn:tildeUn} results in the splitting
\begin{eqnarray*}
\tilde U_0 - \tilde u_h(0) &=  \Big((\gamma I + F_{h,\tau}^N)^{-1} (P_h-R_h)g  +  (\gamma I + F_h(T))^{-1} (R_h-P_h)g\Big) \\
&\quad + \Big((\gamma I + F_{h,\tau}^N)^{-1}  -  (\gamma I + F_h(T))^{-1}\Big)R_h g\\
& = I_1 + I_2.
\end{eqnarray*}
Using the approximation property of $P_h$ and $R_h$, Lemma \ref{lem:tildeuh-0}, Corollary \ref{cor:fully-sol-2}, and the regularity result in Theorem \ref{thm:reg-u},
we have an estimate of the term $I_1$:
\begin{equation*}
 \| I_1 \|_{L^2\II}\le c h^2 \gamma^{-1} \| u_0 \|_{L^2\II}.
\end{equation*}
To bound the term $I_2$, we note that
\begin{eqnarray*}
 \| I_2 \|_{L^2\II}^2 &= \sum_{j=1}^K \Big[\frac{1}{\gamma  + F_{\tau}^N(\lambda_j^h)}-
 \frac{1}
{\gamma  + E_{\alpha,1}(-\lambda_j^h T^{\alpha})}\Big]^2 (R_h g, \fy_j^h)^2 \\
&=\sum_{j=1}^K
 \Big|\frac{ [E_{\alpha,1}(-\lambda_j^hT^{\alpha})-F_{\tau}^N(\lambda_j^h)](\lambda_j^h)^{-1}}
{(\gamma  + F_{\tau}^N(\lambda_j^h))(\gamma  + E_{\alpha,1}(-\lambda_j^h T^{\alpha}))}\Big|^2 (\lambda_j^h)^{2} (R_h g, \fy_j^h)^2.
\end{eqnarray*}
Then we apply Lemma \ref{lem:fully-sol} to obtain
\begin{equation}\label{eqn:der-1}
 \hskip-35pt \|I_2\|_{L^2\II}^2
 \le c \tau^{2}\gamma^{-2}\sum_{j=1}^K
 \Big|\frac{1}
{(\gamma  + E_{\alpha,1}(-\lambda_j^h T^{\alpha}))(\lambda_j^h)^{q/2}}\Big|^2 (\lambda_j^h)^{2+q} (R_h g, \fy_j^h)^2.
\end{equation}
For $q=0$, we use Lemma \ref{lem1} to deduce that
\begin{eqnarray*}
 \| I_2 \|_{L^2\II}^2
 &\le c\tau^{2}\gamma^{-4} \sum_{j=1}^K (\lambda_j^h)^{2} (R_h g, \fy_j^h)^2
 = c\tau^{2}\gamma^{-4} \|  \Delta_h R_h g \|_{L^2\II}^2 .
\end{eqnarray*}
Using fact that $P_h \Delta= \Delta_h R_h$ and applying Theorem \ref{thm:reg-u}, we obtain
\begin{eqnarray*}
\hskip-30pt\| I_2 \|_{L^2\II}^2
&= c\tau^{2}\gamma^{-4}  \| P_h \Delta g \|_{L^2\II}^2=c\tau^{2}\gamma^{-4}  \| \Delta g \|_{L^2\II}^2
\le c\tau^{2}\gamma^{-4} T^{-\alpha} \| u_0 \|_{L^2\II}.
\end{eqnarray*}

Next we turn to the case that $q=2$. The estimate \eref{eqn:der-1} and Lemma \ref{lem1} imply that
\begin{eqnarray*}
 \| I_2\|_{L^2\II}^2
&\le c\tau^{2}\gamma^{-2}\sum_{j=1}^K
 \Big|\frac{1}
{ E_{\alpha,1}(-\lambda_j^h T^{\alpha}) \lambda_j^h}\Big|^2 (\lambda_j^h)^{4} (R_h g, \fy_j^h)^2\\
&\le c\tau^{2}\gamma^{-2} \sum_{j=1}^K  (\lambda_j^h)^{4} (R_h g, \fy_j^h)^2 = c\tau^{2}\gamma^{-2} \| \Delta_h^2 R_h g \|_{L^2\II}.
\end{eqnarray*}
Now we use the fact that $P_h \Delta= \Delta_h R_h$ and triangle's inequality to derive
\begin{eqnarray}\label{eqn:der-2}
\eqalign{
  \| \Delta_h^2 R_h g \|_{L^2\II} &=  \| \Delta_h P_h \Delta g \|_{L^2\II} \\&\le \| \Delta_h (P_h-R_h) \Delta g \|_{L^2\II} + \| \Delta_h R_h \Delta g\|_{L^2\II}.}
\end{eqnarray}
The second term in \eref{eqn:der-2} can be bounded by
\begin{eqnarray}\label{eqn:der-3}
\eqalign{
  \| \Delta_h R_h \Delta g\|_{L^2\II} &=   \| P_h\Delta^2 g\|_{L^2\II} \\&=  \|  \Delta^2 g\|_{L^2\II}  = \| g \|_{\dH4} \le cT^{-\alpha} \| u_0 \|_{\dH2}.}
\end{eqnarray}
while the first term in \eref{eqn:der-2} can be bounded by using the standard inverse inequality and the approximation properties
\eref{eqn:err-Ph} and \eref{eqn:err-Rh} as
\begin{equation}\label{eqn:der-4}
\eqalign{
\| \Delta_h (P_h-R_h) \Delta g \|_{L^2\II} &\le ch^{-2}\| (P_h-R_h) \Delta g \|_{L^2\II} \\&\le c\| \Delta g \|_{H^2\II} \le cT^{-\alpha}\| u_0 \|_{\dH2}.}
\end{equation}
This leads to the desired estimate with $q=2$. Finally, the estimate for $q\in(0,2)$ follows immediately from interpolation.
\end{proof}

Using the similar argument, one can also derive an estimate of $  \tilde U_n - \tilde u_h(t_n)$ for $n\ge1$.
\begin{lemma}\label{lem:fully-lem3}
Let  $\tilde U_n $  and $\tilde u_h(t)$ be solutions to  \eref{eqn:fully-back-2} and \eref{DBWP-1}, respectively.
Then there holds that
$$\eqalign{\| \tilde U_n - \tilde u_h(t_n)  \|_{L^2\II} \le  c \Big(&\gamma^{-(1-q/2)}  (\tau t_{n}^{\alpha-1}+  \tau \min (\gamma^{-1}, t_n^{-\alpha})) \|  u_0 \|_{\dH{q}}\\
&+h^2\min(\gamma^{-1}, t_n^{-\alpha})\| u_0 \|_{L^2\II}\Big)}.$$
where the generic constant $c$ is independent of  $\gamma$, $\delta$, $\tau$, $n$ and $h$.
\end{lemma}
\begin{proof}
First of all, we split $\tilde U_n - \tilde u_h(t_n)$ into two terms
\begin{eqnarray*}
\tilde U_0 - \tilde u_h(0) &=  \Big(F_{h,\tau}^n(\gamma I + F_{h,\tau}^N)^{-1} (P_h-R_h)g  + F_h(t_n) (\gamma I + F_h(T))^{-1} (R_h-P_h)g\Big) \\
&\quad + \Big(F_{h,\tau}^n(\gamma I + F_{h,\tau}^N)^{-1}  - F_h({t_n}) (\gamma I + F_h(T))^{-1}\Big)R_h g\\
& = I_1 + I_2.
\end{eqnarray*}
The approximation property of $P_h$ and $R_h$, Theorem \ref{thm:reg-u}, Lemma \ref{lem:tildeuh-0} and Corollary \ref{cor:fully-sol-2}
lead to an estimate of the term $I_1$:
\begin{equation*}
 \| I_1 \|_{L^2\II}\le c h^2 \min(\gamma^{-1},t_n^{-\alpha}) \| u_0 \|_{L^2\II}.
\end{equation*}

Next, we turn to the $I_2$, which can be split into  three components:
\begin{eqnarray*} 	 	
\hskip-0.3in \| I_2 \|_{L^2\II}^2 &= \sum_{j=1}^K \Big[\frac{F_{\tau}^n(\lambda_j^h)}{\gamma  + F_{\tau}^N(\lambda_j^h)}-
 \frac{E_{\alpha,1}(-\lambda_j^h t_n^{\alpha})}
{\gamma  + E_{\alpha,1}(-\lambda_j^h T^{\alpha})}\Big]^2 (R_h g, \fy_j^h)^2 \\
&\leq c\sum_{j=1}^K \big|\frac{\gamma[F_\tau^n(\lambda_j^h)-E_{\al,1}(-\lambda_j^h {t_n^\al})](\lambda_j^h)^{-1}}
{(\gamma+F_\tau^N(\lambda_j^h))(\gamma+E_{\al,1}(-\lambda_j^h T^\al)}\big|^2(\lambda_j^h)^2(R_h g, \fy_j^h)^2\\
&+c\sum_{j=1}^K \big|\frac{F_\tau^N(\lambda_j^h)[F_\tau^n(\lambda_j^h)-E_{\al,1}(-\lambda_j^h{t_n^\al})](\lambda_j^h)^{-1}}
{(\gamma+F_\tau^N(\lambda_j^h))(\gamma+E_{\al,1}(-\lambda_j^h T^\al)}\big|^2(\lambda_j^h)^2(R_h g, \fy_j^h)^2\\
&+c\sum_{j=1}^K \big|\frac{F_\tau^n(\lambda_j^h)[(E_{\al,1}(-\lambda_j^hT^\al)-F_\tau^N(\lambda_j^h)](\lambda_j^h)^{-1}}
{(\gamma+F_\tau^N(\lambda_j^h))(\gamma+E_{\al,1}(-\lambda_j^h T^\al)}\big|^2(\lambda_j^h)^2(R_h g, \fy_j^h)^2\\
&=:\sum_{k=1}^3 I_{2,k}.
\end{eqnarray*}
The estimates of $I_{2,1}$ and $I_{2,2}$ follows directly from the proof of Lemma \ref{lem:fully-lem2}, i.e.,
$$ I_{2,1} + I_{2,2} \le  c\tau^2 t_{n}^{2\alpha-2} \gamma^{-(2-q)}\| u_0 \|_{\dH q} .$$
Now it remains to bound $I_3$. Here we apply Lemma \ref{lem:fully-sol} and Corollary \ref{cor:fully-sol-2}, and obtain
\begin{equation*} 	 	
\hskip-0.7in  I_{2,3} \leq  c\tau^2T^{2\al-2}\min\{\gamma^{-2},{t_n^{-2\al}}\}\sum_{j=1}^K \big|\frac{1}
{(\gamma+E_{\al,1}(-\lambda_j^h T^\al))(\lambda_j^h)^{q/2}}\big|^2(\lambda_j^h)^{2+q}(R_h g, \fy_j^h)^2.
\end{equation*}
Then the estimates \eref{eqn:der-1}--\eref{eqn:der-4} imply
\begin{equation*}
I_{2,3} \leq c \tau^2 \gamma^{-(2-q)}\min\{\gamma^{-2},t_n^{-2\al}\}\| u_0 \|_{\dH q}^2 .
\end{equation*}
This completes the proof of the lemma.
\end{proof}

Then Lemmas \ref{lem:fully-lem1}--\ref{lem:fully-lem3} together with Theorem \ref{thm:nonsmooth} and Corollary \ref{cor:err-tildeu-L2}
result in the main theorem of this section.
\begin{theorem}\label{thm:fully-error-rev}
Let $u$ be the solution to the backward subdiffusion problem \eref{eqn:fde}, and $\tilde U_n^\delta $  be the solution to the (regularized) fully discrete scheme \eref{eqn:fully-back}.
Then we have the following error estimate:
\begin{itemize}
\item[(a)] In case that $u_0\in \dH2$, there holds
\begin{eqnarray*}
\fl\| \tilde U_n^\delta - u(t_n)  \|_{L^2\II} \le c\left\{\eqalign{
& \gamma+(h^2+\tau+\delta)\min(\gamma^{-1}, t_n^{-\alpha})+
\tau t_{n}^{\alpha-1}  , n\ge1; \\
&\gamma+(h^2+\tau+\delta)\gamma^{-1},\qquad n=0.}\right.
\end{eqnarray*}
\item[(b)] In case that $u_0\in L^2\II$, there holds for $n\ge1$
\begin{equation*}
\fl\| \tilde U_n^\delta - u(t_n)  \|_{L^2\II} \le c\Big( \gamma t_{n}^{-\alpha}+\big(\delta+
  \gamma^{-1}(h^2 \ell_h +\tau)\big)\min(\gamma^{-1},t_n^{-\alpha}) + \gamma^{-1}\tau t_n^{\alpha-1}\Big).
 \end{equation*}
Meanwhile, for $n=0$, there holds
\begin{equation*}
 \fl \|\tilde U_0^\delta-u(0)\|_{L^2\II} \rightarrow 0,\quad \mbox{as}~~\gamma\rightarrow0, ~\frac\delta\gamma\rightarrow0,
  ~ \frac{h\ell_h^{\frac12}}{\gamma}\rightarrow0~\mbox{and}~ \frac{\tau^{\frac12}}{\gamma}\rightarrow0.
\end{equation*}
\end{itemize}
\end{theorem}

{
\begin{remark}\label{rem:fully-interm}
For the intermediate case that $u_0 \in \dH q$, $q\in(0,2)$, the error estimate  follows  from Lemma \ref{lem:fully-lem1}--\ref{lem:fully-lem3},
Theorem \ref{thm:nonsmooth}, and the real interpolarion. In particular, for $n=0$, we have
\begin{equation*}
\| \tilde U_0^\delta - u(0)  \|_{L^2\II} \le c\Big( \gamma^{q/2} + \delta \gamma^{-1}
+\gamma^{-2+q/2}(h^2\ell_h^{1-q/2} +\tau) \Big).
 \end{equation*}
Then one may obtain the optimal convergence rate $O(\delta^{\frac{q}{q+2}})$ by the \textit{a priori} choices: $$\gamma=O(\delta^{\frac{2}{q+2}}),~~
h\ell_h^{\frac12-\frac{q}4} = O(\delta^{\frac{2}{q+2}})~~\mbox{and}~ ~\tau=O(\delta^{\frac{2}{q+2}}).$$
Meanwhile, for $n\ge1$, there holds the estimate
\begin{eqnarray*}
\hskip-0.8in\| \tilde U_n^\delta - u(t_n)  \|_{L^2\II} &\le  c\Big(  \min(\gamma^{q/2}, \gamma t_n^{-(1-q/2)\alpha})
+\Big(\gamma^{-(1-q/2)}(h^2\ell_h^{1-q/2}+\tau) + \delta\Big)\min(\gamma^{-1},t_n^{-\alpha})\\
&\qquad+\gamma^{-(1-q/2)}  \tau t_{n}^{\alpha-1} \Big).
\end{eqnarray*}
Asymptotically, the \textit{a priori} choice, that $\gamma=O(\delta)$, $h\ell_h^{\frac12-\frac{q}4}=O(\delta^{1-\frac{q}4})$ and $\tau=O(\delta^{2-\frac{q}2})$,
leads to the optimal convergence rate $O(\delta)$.
\end{remark}}

{
\begin{remark}\label{rem:fully-2}
Theorem \ref{thm:fully-error-rev} and Remark \ref{rem:fully-interm} indicates the correct way
to scale noise level $\delta$, regularization parameter $\gamma$, and mesh sizes $h$ and $\tau$, with different types of problem data.
The novel argument uses the smoothing properties of fully discrete solution operators,
and the nonstandard error estimate for the direct problem
\cite{JinLiZhou:2017sisc, JinLiZhou:var}.
\end{remark}}

\section{Numerical results}\label{sec:numerics}
In this section, we shall illustrate the theoretical results by presenting some {1-D and 2-D} examples. Throughout, we consider the observation data
\begin{equation*}
g_\delta=u(T)+\varepsilon \delta \sup_{x\in\Omega}u(x,T),
\end{equation*}
$\varepsilon$ is generated following the standard Gaussian distribution and $\delta$ denotes the (relative) noise level.
Throughout this section, we fix $T=1$.

We consider the one-dimensional subdiffusion problem in the unit interval  $\Omega=(0,1)$.
We use the standard piecewise linear FEM with uniform mesh size $h=1/(K+1)$
for the space discretization, and the CQ-BE method with uniform step size $\tau=T/N$ for the time discretization.
Although the fully discrete solution can be efficiently computed by using conjugate gradient method, {in 1-D example}
we apply the following direct method by spectral decomposition to avoid any iteration error.

For the uniform mesh size $h=1/(K+1)$, the eigenparis of $-\Delta_h$ has the closed form:
\begin{equation}\label{eqn:semi-eigenpair}
\lambda^h_j=\frac{6}{h^2}\frac{1-\cos(j\pi h)}{2+\cos(j\pi h)},
\quad \fy^h_j(x_i)=\sqrt{2}\sin(j\pi x_i),\quad i,j=1,2,\cdots,K.
\end{equation}
The semidiscrete solution of the forward problem can be computed by using the solution representation \eref{Duhamel_o}
involving the Mittag-Leffler function \eref{eqn:ML}, which could be evaluated by the
algorithm developed in \cite{SeyboldHilfer:2008}.
We compute the observation data $u(T)$ and reference solution $u(t)$ with $t\in[0,T)$
by using the semidiscrete scheme with a very fine mesh size, i.e., $h=1/2000$.

For each example, we measure the accuracy of the approximation $\tilde{u}_h^\delta (t)$ and $\tilde{U}_n^\delta$ by the normalized error
$\| u(t)-\tilde{u}_h^\delta(t)\|_{L^2\II}/\| u(t) \|_{L^2\II}$ and $\| u(t_n)-\tilde{U}_n^\delta \|_{L^2\II} / \|u(t_n)\|_{L^2\II}$.
The normalization enables us to observe the behaviour of the error with respect to $\alpha$ and $t$.
 \vskip5pt

\paragraph{\bf Example (a): Smooth initial data.} We start with the smooth initial condition
$$u_0(x)=x(1-x) \in \dH2 = H^2\II\cap H_0^1\II,$$
and source term $f\equiv0$. We compute the solution of the (regularized) semidiscrete scheme \eref{DBWP} by
\begin{equation}\label{eqn:numer-01}
 \widetilde u_h^\delta (t) = \sum_{j=1}^{K}  \frac{E_{\alpha,1}(\lambda_{j}^h t^\alpha)}{\gamma + E_{\alpha,1}(\lambda_{j}^h T^\alpha)} (g_\delta,\fy_j^h){\varphi_j^h},
 \end{equation}
where the eigenpairs $(\lambda_{j}^h, \fy_{j}^h)$, for $j=1,\ldots,K-1$, are given by  \eref{eqn:semi-eigenpair}.
%
In Figure \ref{semiplot}, we plot the error of numerical solution \eref{eqn:numer-01}, with different fractional order $\alpha$
and at different time.
By Theorem \ref{thm:err-smooth} and  Remark \ref{rem:rate-semi-smooth},  we compute the $\tilde u_h^\delta(0)$ with
$h=\sqrt{\delta}$, $\gamma=\sqrt\delta$ for a given $\delta$; and compute the $\tilde u_h^\delta(t)$ for $t>0$ with
$h=\sqrt{\delta}$, $\gamma=\delta$ for a given $\delta$. Numerical experiments show an empirical convergence rate of $\Or(\sqrt{\delta})$ for $t=0$,
and $\Or(\delta)$ for $t>0$.
This  coincides with our theoretical result (Theorem \ref{thm:err-smooth}).

\begin{figure}[htbp]
\centering
\includegraphics[scale=0.3]{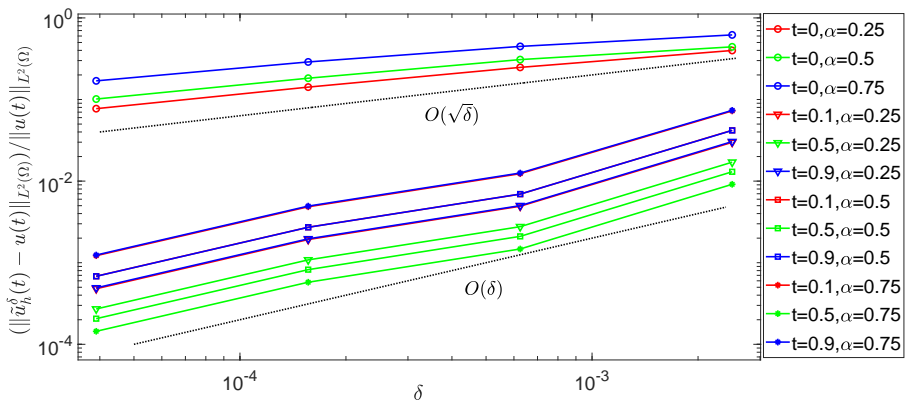}
\caption[plain]{Plot of $\| u(t) - \tilde u_h^\delta (t)  \|_{L^2\II}/\| u(t)\|_{L^2\II}$ with $h=\gamma = \sqrt\delta$ for $t=0$;\\ and
$h=\sqrt\delta$, $\gamma = \delta$ for $t_n>0$. }\label{semiplot}
\end{figure}

In Figure \ref{fullyplot}, we plot the error of numerical reconstruction by the fully scheme
\eref{eqn:fully-back}, with different $\alpha$ and at different time. In our experiments, we compute fully discrete solution $\tilde U_n^\delta$  by
$$  \tilde U_n^\delta = \sum_{j=1}^{K} \frac{F_{\tau}^n(\lambda_j^h)}{\gamma  + F_{\tau}^N(\lambda_j^h)} (P_h g_\delta, \fy_j^h){\fy_j^h}.$$
Then Theorem \ref{thm:fully-error-rev} (i) implies for $u_0\in \dot{H}^2\II$
\begin{eqnarray*}
\fl\| \tilde U_n^\delta - u(t_n)  \|_{L^2\II} \le c\left\{\eqalign{
& \gamma+(h^2+\tau+\delta)\min(\gamma^{-1}, t_n^{-\alpha})+
\tau t_{n}^{\alpha-1} , n\ge1; \\
&\gamma+(h^2 +\tau+\delta)\gamma^{-1}  ,\qquad n=0.}\right.
\end{eqnarray*}
For $t=0$, we let $h=\gamma=\sqrt{\delta}$ and $\tau={\delta}$, and then we observe that the empirical convergence rate is $\Or(\sqrt{\delta})$.
Meanwhile, for $t>0$, and we let $h=\sqrt{\gamma}=\sqrt{\delta}=\sqrt{\tau}$.
The empirical convergence rate is $\Or(\delta)$.
These observation agrees well with our theoretical results in Theorem  \ref{thm:fully-error-rev} (i).

\begin{figure}[htbp]
\centering
\includegraphics[scale=0.3]{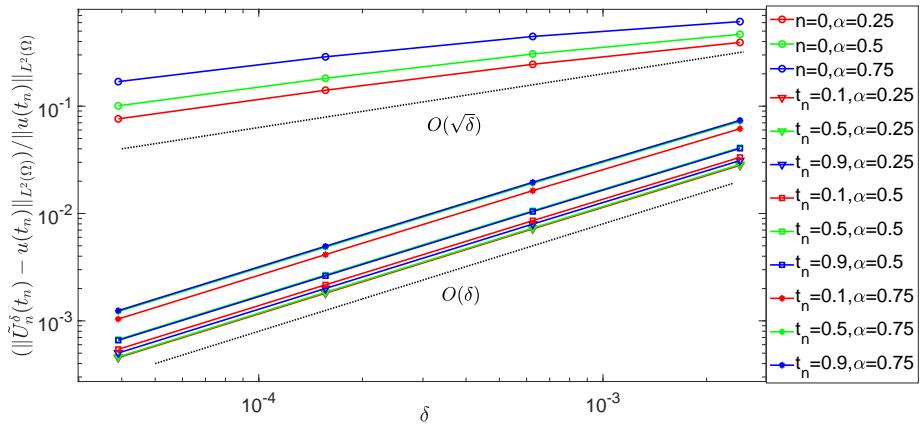}
\caption[plain]{Plot of $\| u(t_n) - \tilde U_n^\delta (t_n)  \|_{L^2\II}/\| u(t_n)\|_{L^2\II}$ with $h=\sqrt\delta$, $\tau = \delta$ \\and $\gamma =  \sqrt\delta$ for $t_n=0$; and
$h=\sqrt\delta$, $\tau=\delta$, $\gamma = \delta$ for $t_n>0$.}\label{fullyplot}
\end{figure}

\paragraph{\bf Example (b): Nonsmooth initial data.}
Now we test  numerical experiments with a step initial condition:
\begin{equation*}
u_0(x)=\left\{\eqalign{0,\,\,\,\, 0\leq x\leq \frac{1}{2},\\
1,\,\,\,\, \frac{1}{2}< x< 1 .
}\right.
\end{equation*}
Since $u_0$ is discontinuous and piecewise smooth,
it is easy to see that $u_0\in H^{\frac12-\epsilon}\II$ for any $\epsilon\in(0,\frac12]$.

According to Theorem \ref{thm:nonsmooth}, we have the error estimate of the semidiscrete solution at $t=0$:
\begin{equation*}
\hskip-0.6in\|\tildu_h(t)-u(t)\|_{L^2\II}\le c\big(  \gamma^{q/2}
+h^2\ell_h^{1-q/2} \gamma^{-(2-q/2)} + \delta  \gamma^{-1}\big),\quad \mbox{with}~~ u_0\in \dH q.
\end{equation*}
This implies that the convergence rate
may deteriorate when the initial data gets worse.
This is fully supported by empirical results showed in Table \ref{table1},
where we present the $L^2$-error of the semidiscrete solution at $t=0$.
In the computation, we let $h=O(\delta^{\frac45})$ and $\gamma=O(\delta^{\frac45})$
in order to balance to noise level, regularization parameter
and the discretization error. Then the empirical convergence rate
is $O(\delta^\frac15)$, which is consistent with the theoretical results. 

Meanwhile, for a fixed $t>0$, we have the error estimate (cf. Theorem \ref{thm:nonsmooth})
\begin{equation*}
\|\tildu_h(t)-u(t)\|_{L^2\II}\le c\big( \gamma \,t^{ q\alpha/2}
+\gamma^{-(1-q/2)}h^2\ell_h^{1-q/2} + \delta \big)  t^{-\alpha}.
\end{equation*}
This implies the almost optimal scaling $h=O(\delta^{\frac78})$ and $\gamma=O(\delta)$,
and the resulting optimal convergence rate $O(\delta)$.
This is supported by the numerical results shown in Table \ref{table2}.

For the numerical reconstruction by the fully discrete scheme \eref{eqn:fully-back},
we recall the result in Remark \ref{rem:fully-interm}. To compute $\tilde U^\delta_0$,
we let $\gamma=O(\delta^{\frac45})$, $h=O(\delta^\frac45)$ and $\tau=O(\delta^\frac85)$, for a given $\delta$.
Then our theory indicates an convergence rate of $O(\delta^{\frac15})$, which agrees well with the numerical results
in Table \ref{table3}.
On the other hand, to compute $\tilde U^\delta_n$ for a fixed $t_n>0$ and $\delta>0$, we let $h = \delta^{\frac74}$, $\tau=O(\delta^{\frac78})$
and $\gamma=O(\delta)$. Then the empirical convergence rate is close to $O(\delta)$,
which fully supports our theoretical estimates {in Table \ref{table4}}.

\begin{table}[htbp]
\centering
\caption{Example (b): error of $\tilde{u}_h^\delta(0)$, with $\delta=1/M$, $h=\gamma=\delta^{\frac45}$.}\label{table1}\vskip5pt
\begin{tabular}{|c|cccc|c|}
\hline
$\al\backslash M$&40&80&160&320&\mbox{Rate}($\delta$)\\
\hline
0.25&4.68e-1&4.07e-1&3.48e-1&2.95e-1&0.22(0.20)\\
0.5&5.07e-1&4.46e-1&3.84e-1&3.27e-1&0.21(0.20)\\
0.75&5.70e-1&5.18e-1&4.59e-1&3.98e-1&0.17(0.20)\\
\hline
\end{tabular}
\end{table}

\begin{table}[htbp]
\centering
\caption{Example (b):  error of $\tilde{u}_h^\delta(t)$ at different $t$ with $\delta=1/M$, $h=\delta^{\frac78}$, $\gamma=\delta/5$.}\label{table2}\vskip5pt
\begin{tabular}{|c|c|cccc|c|}
\hline
$\al$&$t \backslash M$&40&80&160&320&\mbox{Rate}($\delta$)\\
\hline
   &0.1&7.91e-3&4.34e-3&2.30e-3&1.20e-3&0.91(1.00)\\
0.5&0.5&3.51e-3&1.93e-3&1.02e-3&5.33e-4&0.91(1.00)\\
   &0.9&2.41e-3&1.33e-3&7.13e-4&3.73e-4&0.90(1.00)\\
\hline
\end{tabular}
\end{table}

\begin{table}[htbp]
\centering
\caption{Example (b): error of $\tilde{U}_0^\delta$, with $\delta=1/M$, $h=\gamma=\delta^{\frac45}$, $\tau=\delta^{\frac85}$.}\label{table3}\vskip5pt
\begin{tabular}{|c|cccc|c|}
\hline
$\al \backslash M$&40&80&160&320&\mbox{Rate}($\delta$)\\
\hline
0.25 &4.70e-1&4.07e-1&3.48e-1&2.96e-1&0.22(0.20)\\
0.5 &5.08e-1&4.47e-1&3.85e-1&3.28e-1&0.21(0.20)\\
0.75 &5.70e-1&5.17e-1&4.59e-1&3.98e-1&0.17(0.20)\\
\hline
\end{tabular}
\end{table}

\begin{table}[htbp]
\centering
\caption{Example(b): error of $\tilde{U}_n^\delta$, with $\delta=1/M$, $h=\delta^{\frac78}$, $\tau=\delta^{\frac74}$, and $\gamma= \delta/5$.}\label{table4}\vskip5pt
\begin{tabular}{|c|c|cccc|c|}
\hline
$\al$&$t_n\backslash M$&40&80&160&320&\mbox{Rate}($\delta$)\\
\hline
   &0.1&6.76e-3&3.82e-3&2.06e-3&1.08e-3&0.88(1.00)\\
0.5&0.5&3.46e-3&1.90e-3&1.01e-3&5.24e-4&0.91(1.00)\\
   &0.9&2.55e-3&1.40e-3&7.47e-4&3.89e-4&0.90(1.00)\\
\hline
\end{tabular}
\end{table}	
\newpage

{
\paragraph{\bf Example (c): 2D problem.}
Now we consider a two-dimensional problem in a unit square domain $\Omega=(0,1)^2$. We choose the smooth initial condition
\begin{equation*}
u_0(x,y)=x(1-x)y(1-y)\in \dot{H}^2(\Omega),
\end{equation*}
and zero source term $f\equiv0$. In the computation, we divided $\Omega$
 into regular right triangles with $K$ equal
subintervals of length $h=1/K$ on each side of the domain.
Here, we apply the conjugate gradient method to numerically solve the
discrete system, instead of the direct approach by the spectral decomposition 
in Example (a) and (b).
 
For $t=0$, we let $h=\gamma=\sqrt{\delta}=\sqrt{\tau}$, and we observe that the
convergence rate is $O(\sqrt{\delta})$, see Table \ref{table5}).
Moreover, In Table \ref{table6}, we test the convergence rate for $t=T/2$. By letting $h=\sqrt{\gamma}=\sqrt{\delta}=\sqrt{\tau}$,
the experiments show that the convergence rate is $O(\delta)$.
All emperical results agree well with our theoretical finding in Theorem \ref{thm:fully-error-rev}.
 
\begin{table}[htbp]
\centering
\caption{Example(c): error of $\tilde{U}_0^\delta$, with $\delta=1/M$, $h=\sqrt{\delta}$, $\tau=\delta$, and $\gamma= \sqrt{\delta}$.}\label{table5}\vskip5pt
\begin{tabular}{|c|cccc|c|}
\hline
$\al\backslash M$&800&1600&3200&6400&\mbox{Rate}($\delta$)\\
\hline
0.25  &1.27e-2&9.57e-3&6.61e-3&3.96e-3&0.56(0.50)\\
0.5   &1.57e-2&1.27e-2&9.53e-3&6.57e-3&0.42(0.50)\\
0.75  &2.28e-3&1.96e-3&1.57e-3&1.11e-3&0.34(0.50)\\
\hline
\end{tabular}
\end{table}	

\begin{table}[htbp]
\centering
\caption{Example(c): error of $\tilde{U}_n^\delta$, with $t_n=T/2$, $\delta=1/M$, $h=\sqrt{\delta}$, $\tau=\delta$, and $\gamma= \delta$.}\label{table6}\vskip5pt
\begin{tabular}{|c| cccc|c|}
\hline
$\al \backslash M$&800&1600&3200&6400&\mbox{Rate}($\delta$)\\
\hline
0.25 & 5.09e-5&2.59e-5&1.31e-5&6.59e-6&0.98(1.00)\\
0.5  & 6.00e-5&3.08e-5&1.56e-5&7.90e-6&0.98(1.00)\\
0.75 & 7.06e-5&3.71e-5&1.89e-5&9.55e-6&0.96(1.00)\\
\hline
\end{tabular}
\end{table}}
%


\section*{Acknowledgements}
This project is partially supported by  a Hong Kong RGC grant (project no. 25300818).

\end{document}